\documentclass[11pt,english]{article}
\usepackage[T1]{fontenc}
\usepackage[latin9]{inputenc}
\usepackage[letterpaper]{geometry}
\usepackage{verbatim}
\usepackage{mathrsfs}
\usepackage{amsmath}
\usepackage{amsthm}
\usepackage{amssymb}
\usepackage{graphicx}
\usepackage{setspace}
\usepackage{color}
\usepackage{cite}

\makeatletter

\providecommand{\tabularnewline}{\\}

\numberwithin{equation}{section}
\numberwithin{figure}{section}
\theoremstyle{plain}
\newtheorem{thm}{\protect\theoremname}[section]
  \theoremstyle{plain}
  \newtheorem{prop}[thm]{\protect\propositionname}
  \theoremstyle{plain}
  \newtheorem{lem}[thm]{\protect\lemmaname}
  \theoremstyle{plain}
  \newtheorem{cor}[thm]{\protect\corollaryname}

\numberwithin{equation}{section} 
\usepackage{graphicx}
\graphicspath{{figs/}}

\newcommand{\ddt}[1]{\frac{\partial}{\partial t}#1}
\newcommand{\inp}[2][]{\left(#1, #2\right)}
\newcommand{\gnp}[2][]{\langle#1, #2\rangle}
\def\<{\langle}
\def\>{\rangle}

\def\Tc{\mathcal{T}}

\def\X{\mathbb{X}}
\def\W{\mathbb{Q}}

\def\R{\mathbb{R}}
\def\M{\mathbb{M}}
\def\S{\mathbb{S}}
\def\N{\mathbb{N}}

\def\a{\alpha}
\def\s{\sigma}
\def\t{\tau}
\def\g{\gamma}

\def\l{\lambda}
\def\lah{\lambda_h}
\def\m{\mu}

\def\Pit{\tilde{\Pi}}

\def\r{\mathbf{r}}

\def\tr{\operatorname{tr}}

\def\dvr{\operatorname{div}}  

\def\dt{\partial_t}

\def\Dt{\Delta t}

\def\O{\Omega}
\def\G{\Gamma}
\def\Gn{\Gamma_N}
\def\Gd{\Gamma_D}
\def\dO{{\partial\Omega}}

\def\l2o{{L^{2}(\Omega)}}

\def\K{{K^{-1}}}
\def\ud{\dot{u}}

\newtheorem{remark}{Remark}[section]

\def\ss{\sigma_{h,i}^{*,n+1}}

\def\zs{z_{h,i}^{*,n+1}}
\def\ps{p_{h,i}^{*,n+1}}

\def\ssj{\sigma_{h,j}^{*,n+1}}

\def\psj{p_{h,j}^{*,n+1}}

\def\sb{\bar{\sigma}_{h,i}^{n+1}}

\def\zb{\bar{z}_{h,i}^{n+1}}
\def\pb{\bar{p}_{h,i}^{n+1}}

\def\usd{\dot{u}_{h,i}^{*,n+1}}
\def\gsd{\dot{\g}_{h,i}^{*,n+1}}

\def\ubd{\bar{\dot{u}}_{h,i}^{n+1}}
\def\gbd{\bar{\dot{\g}}_{h,i}^{n+1}}

\allowdisplaybreaks

\@ifundefined{showcaptionsetup}{}{%
 \PassOptionsToPackage{caption=false}{subfig}}
\usepackage{subfig}
\makeatother

\usepackage{babel}
  \providecommand{\corollaryname}{Corollary}
  \providecommand{\lemmaname}{Lemma}
  \providecommand{\propositionname}{Proposition}
\providecommand{\theoremname}{Theorem}

\begin{document}

\title{Domain decomposition and partitioning methods for mixed finite element
discretizations of the Biot system of poroelasticity}

\author{Manu Jayadharan
  \thanks{Department of Mathematics, University of
    Pittsburgh, Pittsburgh, PA 15260, USA;~{\tt \{manu.jayadharan@pitt.edu, elk58@pitt.edu, yotov@math.pitt.edu\}}. Partially supported by NSF grant DMS 1818775.}
    \and Eldar Khattatov\footnotemark[1]
    \and Ivan Yotov\footnotemark[1]}
  
\maketitle
\begin{abstract}

We develop non-overlapping domain decomposition methods for the Biot
system of poroelasticity in a mixed form. The solid deformation is
modeled with a mixed three-field formulation with weak stress
symmetry. The fluid flow is modeled with a mixed Darcy formulation. We
introduce displacement and pressure Lagrange multipliers on the
subdomain interfaces to impose weakly continuity of normal stress and
normal velocity, respectively. The global problem is reduced to an
interface problem for the Lagrange multipliers, which is solved by a
Krylov space iterative method. We study both monolithic and split
methods.  In the monolithic method, a coupled displacement-pressure
interface problem is solved, with each iteration requiring the
solution of local Biot problems. We show that the resulting interface
operator is positive definite and analyze the convergence of the
iteration. We further study drained split and fixed stress Biot
splittings, in which case we solve separate interface problems
requiring elasticity and Darcy solves.  We analyze the
stability of the split formulations. Numerical experiments are
presented to illustrate the convergence of the domain decomposition
methods and compare their accuracy and efficiency.
\end{abstract}


\section{Introduction}

In this paper we study several non-overlapping domain decomposition 
methods for the Biot system of poroelasticity \cite{biot1941general},
which models the flow of a viscous fluid through a poroelastic medium
along with the deformation of the medium. Such flow occurs in many
geophysics phenomena like earthquakes, landslides, and flow of oil
inside mineral rocks and plays a key role in engineering applications
such as hydrocarbon extraction through hydraulic or thermal
fracturing. We use the classical Biot system of poroelasticity with
the quasi-static assumption, which is particularly relevant in
geoscience applications. The model consists of an equilibrium equation
for the solid medium and a mass balance equation for the flow of the
fluid through the medium. The system is fully coupled, with the fluid
pressure contributing to the solid stress and the divergence of the
solid displacement affecting the fluid content.

The numerical solution of the Biot system has been extensively studied
in the literature. Various formulations have been considered,
including two-field displacement--pressure formulations
\cite{Gaspar-FD-Biot,murad1992improved,Zik-MINI}, three-field
displacement--pressure--Darcy velocity formulations
\cite{phillips2007coupling1,Zik-stab,Yi-Biot-locking,
  Zik-nonconf,Lee-Biot-three-field,Yi-Biot-nonconf,phillips-DG,Whe-Xue-Yot-Biot}, and
three-field displacement--pressure--total pressure formulations
\cite{Lee-Mardal-Winther,ORB}. More recently, fully-mixed formulations
of the Biot system have been studied. In \cite{Yi-Biot-mixed}, a
four-field stress--displacement--pressure--Darcy velocity mixed
formulation is developed. A posteriori error estimates for
this formulation are obtained in \cite{Ahmed-apost-CMAME}.
In \cite{Lee-Biot-five-field}, a weakly
symmetric stress--displacement--rotation elasticity formulation is
considered, which is coupled with a mixed pressure--Darcy velocity
flow formulation. Fully-mixed finite element approximations carry the
advantages of local mass and momentum conservation, direct computation
of the fluid velocity and the solid stress, as well as robustness and
locking-free properties with respect to the physical parameters. Mixed
finite element (MFE) methods can also handle discontinuous full tensor
permeabilities and Lam\'{e} coefficients that are typical in
subsurface flows. In our work we focus on the five-field
weak-stress-symmetry formulation from \cite{Lee-Biot-five-field},
since weakly symmetric MFE methods for elasticity allow for reduced
number of degrees of freedom. Moreover, a multipoint
stress--multipoint flux mixed finite element approximation for this
formulation has been recently developed in \cite{msfmfe-Biot}, which
can be reduced to a positive definite cell-centered scheme for
pressure and displacement only, see also a related finite volume
method in \cite{Jan-SINUM-Biot}. While our domain decomposition
methods are developed for the weakly symmetric formulation from
\cite{Lee-Biot-five-field}, the analysis carries over in a
straightforward way to the strongly symmetric formulation from
\cite{Yi-Biot-mixed}.

Discretizations of the Biot system of poroelasticity for practical
applications typically result in large algebraic systems of
equations. The efficient solution of these systems is critical for the
ability of the numerical method to provide the desired resolution. In
this work we focus on non-overlapping domain decomposition methods
\cite{Toselli-Widlund,QV}. These methods split the computational
domain into multiple non-overlapping subdomains with algebraic systems
of lower complexity that are easier to solve.  A global problem
enforcing appropriate interface conditions is solved iteratively to
recover the global solution. This approach naturally leads to scalable
parallel algorithms. Despite the abundance of works on discretizations
of the Biot system, there have been very few results on domain
decomposition methods for this problem.  In \cite{girault2011domain},
a domain decomposition method using mortar elements for coupling the
poroelastic model with an elastic model in an adjacent region is
presented. In that work, the Biot region is not decomposed into
subdomains. In \cite{HoracioNurbs1,FlorezDD}, an iterative coupling
method is employed for a two-field displacement--pressure formulation,
and classical domain decomposition techniques are applied separately
for the elasticity and flow equations. A monolithic domain
decomposition method for the two-field formulation of poroelasticity
combining primal and dual variables is developed in
\cite{GosseletMonoDD}. To the best of our knowledge, domain
decomposition methods for mixed formulations of poroelasticity have
not been studied in the literature.

In this paper we study both monolithic and split non-overlapping
domain decomposition methods for the five-field fully mixed
formulation of poroelasticity with weak stress symmetry from
\cite{Lee-Biot-five-field,msfmfe-Biot}. Monolithic methods require
solving the coupled Biot system, while split methods only require
solving elasticity and flow problems separately. Both methods have
advantages and disadvantages. Monolithic methods involve solving
larger and possibly ill-conditioned algebraic systems, but may be
better suitable for problems with strong coupling between flow and
mechanics, in which case split or iterative coupling methods may
suffer from stability or convergence issues and require sufficiently
small time steps. Our methods are motivated by the non-overlapping
domain decomposition methods for MFE discretizations of Darcy flow
developed in
\cite{glowinski1988domain,cowsar1995balancing,arbogast2000mixed} and
the non-overlapping domain decomposition methods for MFE
discretizations of elasticity developed recently in
\cite{eldar_elastdd}.

In the first part of the paper we develop a monolithic domain
decomposition method.  We employ a physically heterogeneous Lagrange
multiplier vector consisting of displacement and pressure variables to
impose weakly the continuity of the normal components of stress and
velocity, respectively. The algorithm involves solving at each time
step an interface problem for this Lagrange multiplier vector. We show
that the interface operator is positive definite, although it is not
symmetric in general. As a result, a Krylov space solver such as GMRES
can be employed for the solution of the interface problem. Each
iteration requires solving monolithic Biot subdomain problems with
specified Dirichlet data on the interfaces, which can be done in
parallel.  We establish lower and upper bounds on the spectrum of
the interface operator, which allows us to perform analysis of the
convergence of the GMRES iteration using field-of-values estimates.

In the second part of the paper we study split domain decomposition
methods for the Biot system. Split or iterative coupling methods for
poroelasticity have been extensively studied due to their 
computational efficiency. Four widely used sequential methods are
drained split, undrained split, fixed stress split, and fixed strain
split. Decoupling methods are prone to stability issues and a detailed
stability analysis of the aforementioned schemes using finite volume
methods can be found in \cite{KimDS2011,KimFS2011}, see also
\cite{bukacSplit} for stability analysis of several split methods
using displacement--pressure finite element discretizations.
Iterative coupling methods are based on similar splittings and involve
iterating between the two sub-systems until convergence. Convergence
for non-mixed finite element methods is analyzed in
\cite{Mikelic2013}, while convergence for a four-field mixed finite
element discretization is studied in \cite{YiSplit}. An accelerated
fixed stress splitting scheme for a generalized non-linear
consolidation of unsaturated porous medium is studied in
\cite{RaduAnderson}. Studies of the optimization and acceleration of
the fixed stress decoupling method for the Biot consolidation model,
including techniques such as multirate or adaptive time stepping and
parallel-in-time splittings have been presented in
\cite{Ahmed-FS-JCAM,RaduFS1,RaduFS2,RaduFS3,Almani-multirate}.

In our work we consider drained split (DS) and fixed stress (FS)
decoupling methods in conjunction with non-overlapping domain
decomposition. In particular, at each time step we solve sequentially
an elasticity and a flow problem in the case of DS or a flow and an
elasticity problem in the case of FS splitting. We perform stability
analysis for the two splittings using energy estimates and show that they
are both unconditionally stable with respect to the time step and the
physical parameters. We then employ separate non-overlapping domain
decomposition methods for each of the decoupled problems, using the
methods developed in \cite{glowinski1988domain,arbogast2000mixed} for
flow and \cite{eldar_elastdd} for mechanics. 

In the numerical section we present several computational experiments
designed to verify and compare the accuracy, stability, and
computational efficiency of the three domain decomposition methods for
the Biot system of poroelasticity. In particular, we study the
discretization error and the number of interface iterations, as well
as the effect of the number of subdomains. We also illustrate the
performance of the methods for a physically realistic heterogeneous
problem with data taken from the Society of Petroleum Engineers 
10th Comparative Solution Project.

The rest of the paper is organized as follows. In
Section~\ref{sec:model} we introduce the mathematical model and its
MFE discretization.  The monolithic domain decomposition method is
developed and analyzed in
Section~\ref{subsec:Monolithic-MFE-Domain}. In Section~\ref{sec:split}
we perform stability analysis of the DS and FS decoupling methods and
present the DS and FS domain decomposition methods. The
numerical experiments are presented in Section~\ref{sec:numer},
followed by conclusions in Section~\ref{sec:concl}.

\section{Model problem and its MFE discretization}\label{sec:model}

Let $\Omega\subset\mathbb{R}^{d}$, $d=2,3$ be a simply connected
domain occupied by a linearly elastic porous body. We use the notation
$\M$, $\S$, and $\N$ for the spaces of $\ensuremath{d\times d}$
matrices, symmetric matrices, and skew-symmetric matrices respectively,
all over the field of real numbers, respectively. Throughout the paper,
the divergence operator is the usual divergence for vector fields,
which produces vector field when applied to matrix field by taking
the divergence of each row. For a domain $G
\subset \R^d$, the $L^2(G)$ inner product and norm for scalar, vector,
or tensor valued functions are denoted $\inp[\cdot]{\cdot}_G$ and
$\|\cdot\|_G$, respectively. The norms and seminorms of the Hilbert
spaces $H^k(G)$ are denoted by $\|\cdot\|_{k,G}$ and $| \cdot
|_{k,G}$, respectively. We omit $G$ in the subscript if $G = \O$. For
a section of the domain or element boundary $S \subset \R^{d-1}$ we
write $\gnp[\cdot]{\cdot}_S$ and $\|\cdot\|_S$ for the $L^2(S)$ inner
product (or duality pairing) and norm, respectively. We will also use
the spaces
\begin{align*}
	&H(\dvr; \O) = \{q \in L^2(\O, \R^d) : \dvr q \in L^2(\O)\}, \\
	&H(\dvr; \O, \M) = \{\t \in L^2(\O, \M) : \dvr \t \in L^2(\O,\R^d)\},
\end{align*}    
equipped with the norm
$$
\|\t\|_{\dvr} = \left( \|\t\|^2 + \|\dvr\t\|^2 \right)^{1/2}.
$$
The partial derivative operator with respect to time, $\frac{\partial
}{\partial t}$, is often abbreviated to $\partial_t$.  Finally, $C$
denotes a generic positive constant that is independent of the
discretization parameters $h$ and $\Delta t$.

Given a vector field $f$ representing body
forces and a source term $g$, the quasi-static Biot system of poroelasticity is
\cite{biot1941general}:
\begin{align}
-\dvr\s(u) & =f, & \text{in \ensuremath{\Omega \times (0,T]}},\label{biot-1}\\
\K z+\nabla p & =0, & \text{in \ensuremath{\Omega \times (0,T]}},\label{biot-2}\\
\ddt(c_{0}p+\a\dvr u)+\dvr z & =g, & \text{in \ensuremath{\Omega \times (0,T]}},\label{biot-3}
\end{align}
where $u$ is the displacement, $p$ is the fluid pressure, $z$ is the 
Darcy velocity, and $\sigma$ is the poroelastic stress, defined as
\begin{equation}\label{stress-comb}
\sigma = \sigma_e - \alpha p I.
\end{equation}
Here $I$ is the $d\times d$ identity matrix, $0 < \alpha \le 1$ is the
Biot-Willis constant, and $\sigma_e$ is the elastic stress satisfying the
stress-strain relationship
\begin{equation}\label{stress-strain}
A\s_e = \epsilon(u), \quad \epsilon(u):=\frac{1}{2}\left(\nabla u+(\nabla u)^{T}\right),
\end{equation}
where $A$ is the compliance tensor, which is a symmetric,
bounded and uniformly positive definite linear operator acting from
$\ensuremath{\S\to\S}$, extendible to $\M\to\M$. In the special case of homogeneous
and isotropic body, $A$ is given by,
\begin{equation}
A\sigma=\frac{1}{2\mu}\left(\sigma-\frac{\lambda}{2\mu+d\lambda}\operatorname{tr}(\sigma)I\right),\label{A-defn}
\end{equation}
where $\mu>0$ and $\lambda\ge0$
are the Lam\'e coefficients. In this case, $\sigma_{e}(u)=2\mu\epsilon(u)+\lambda\dvr u\,I$.
Finally, $K$ stands for the permeability
tensor, which is symmetric, bounded, and uniformly positive definite, and
$c_{0} \ge 0$ is the mass storativity.
To close the system, we impose the boundary conditions 
\begin{align}
  u & =g_{u} & \text{on \ensuremath{\Gd^u \times (0,T]}} & , & \s n & =0
    & \text{on \ensuremath{\Gn^\sigma  \times (0,T]}},\label{biot-bc-1}\\
      p & =g_{p} & \text{on \ensuremath{\Gd^p \times (0,T]}} & , & z\cdot n
        & =0 & \text{on \ensuremath{\Gn^z \times (0,T]}},\label{biot-bc-2}
\end{align}
where $\ensuremath{\Gd^u\cup\Gn^\sigma=\Gd^p\cup\Gn^z=\partial\O}$ and 
$n$ is the outward unit normal vector field on $\partial \O$,
along with the initial condition $p(x,0)=p_0(x)$ in $\Omega$.
Compatible initial data for the rest of the variables can be obtained from
($\ref{biot-1}$) and ($\ref{biot-2}$) at $t=0$. 
Well posedness analysis for this system can be found in \cite{SHOWALTER2000310}.

We consider a mixed variational formulation for
($\ref{biot-1}$)--($\ref{biot-bc-2}$) with weak stress symmetry,
following the approach in \cite{Lee-Biot-five-field}. The motivation
is that MFE elasticity spaces with weakly symmetric stress tend to
have fewer degrees of freedom than strongly symmetric MFE
spaces. Moreover, in a recent work, a multipoint stress--multipoint
flux mixed finite element method has been developed for this
formulation that reduces to a positive definite cell-centered scheme
for pressure and displacement only \cite{msfmfe-Biot}. Nevertheless,
the domain decomposition methods in this paper can be employed for
strongly symmetric stress formulations, with the analysis carrying
over in a straightforward way. We introduce a rotation 
Lagrange multiplier $\gamma :=\frac{1}{2}\left(\nabla u - \nabla u^T\right) \in \N$,
which is used to impose weakly symmetry of the stress tensor $\s$. We rewrite 
(\ref{stress-strain}) as
\begin{equation}
A\left( \s + \alpha p I \right)=\nabla u - \gamma \label{stress-strain-non-sym}.
\end{equation}
Combining (\ref{stress-strain}) and (\ref{stress-comb}) gives
$\dvr u = \tr(\epsilon(u)) = \tr(A\sigma_e) = 
\tr A(\s + \alpha p I)$, which can be used to rewrite (\ref{biot-3}) as
\begin{equation}
    \dt(c_{0}p+\a \tr A\left(\s + \alpha p I\right))+\dvr z  =g. \label{biot-3-new} 
\end{equation}
The combination of (\ref{stress-strain-non-sym}), (\ref{biot-1}), (\ref{biot-2}), and
(\ref{biot-3-new}), along with the boundary conditions
\eqref{biot-bc-1}--\eqref{biot-bc-2},
leads to the  variational formulation:
find $\ensuremath{(\sigma,u,\gamma,z,p):[0,T]\to\X\times V\times\W\times Z\times W}$
such that $p(0)=p_0$ and for a.e. $t\in (0,T)$,
\begin{align}
 & \inp[A(\sigma + \a p I)]{\tau} + \inp[u]{\dvr{\tau}}+\inp[\gamma]{\tau}=\gnp[g_{u}]{\t\,n}_{\Gamma_{D}^u}, & \forall\tau\in\X,\label{eq:cts1}\\
 & \inp[\dvr{\sigma}]{v}=-\inp[f]{v}, & \forall v\in V,\label{eq:cts2}\\
 & \inp[\sigma]{\xi}=0, & \forall\xi\in\W,\label{eq:cts3}\\
 & \inp[\K z]{q}-\inp[p]{\dvr{q}}=-\gnp[g_{p}]{q\cdot n}_{\Gamma_{D}^p}, & \forall q\in Z,\label{eq:cts4}\\
  & c_{0}\inp[\dt{p}]{w}+\a\inp[\dt A(\sigma + \a pI)]{wI}
  +\inp[\dvr{z}]{w}=\inp[g]{w}, & \forall w\in W,\label{eq:cts5}
\end{align}
where
\begin{align*}
  &\X = \big\{ \t\in H (\dvr;\Omega,\M) : \t\,n = 0 \text{ on } \Gn^\sigma  \big\},
  \quad V = L^2 (\Omega, \R^d), \quad \W = L^2 (\Omega, \N), \\
  &Z = \big\{ q\in H (\dvr;\Omega) : q\cdot n = 0 \text{ on } \Gn^z  \big\},
  \quad W = L^2 (\Omega).
\end{align*}
It was shown in \cite{msfmfe-Biot} that the system
(\ref{eq:cts1})--(\ref{eq:cts5}) is well posed.

Next, we present the MFE discretization of
(\ref{eq:cts1})--(\ref{eq:cts5}). For simplicity we assume that $\Omega$ is a Lipshicz
polygonal domain. Let $\Tc_{h}$ be a shape-regular quasi-uniform
finite element partition of $\Omega$, consisting of simplices or quadrilaterals,
with $h=\text{max}_{E\in \Tc_{h}} \text{diam}(E)$. The MFE method for solving 
(\ref{eq:cts1})--(\ref{eq:cts5}) is: find
$\ensuremath{(\sigma_{h},u_{h},\gamma_{h},z_{h},p_{h})}:[0,T]\to\X_{h}\times
V_{h}\times\W_{h}\times Z_{h}\times W_{h}$ such that, for a.e. $t\in (0,T)$,
\begin{align}
 & \inp[A(\sigma_{h} + \a p_h I)]{\tau}+\inp[u_{h}]{\dvr{\tau}}+\inp[\gamma_{h}]{\tau}=\gnp[g_{u}]{\t\,n}_{\Gamma_{D}^u}, & \forall\tau\in\X_{h},\label{eq:fe1}\\
 & \inp[\dvr{\sigma_{h}}]{v}=-\inp[f]{v}, & \forall v\in V_{h},\label{eq:fe2}\\
 & \inp[\sigma_{h}]{\xi}=0, & \forall\xi\in\W_{h},\label{eq:fe3}\\
  & \inp[\K z_{h}]{q}-\inp[p_{h}]{\dvr{q}}=-\gnp[g_{p}]{q\cdot n}_{\Gamma_{D}^p},
  & \forall q\in Z_{h},\label{eq:fe4}\\
  & c_{0}\inp[\dt{p_{h}}]{w}+\a\inp[\dt A(\sigma_{h} + \a p_h I)]{wI}
  +\inp[\dvr{z_{h}}]{w}=\inp[g]{w}, & \forall w\in W_{h},\label{eq:fe5}
\end{align}
with discrete initial data obtained as the elliptic projection of the
continuous initial data.  Here $\X_{h}\times V_{h}\times\W_{h}\times
Z_{h}\times W_{h}\subset\X\times V\times\W\times Z\times W$ is a
collection of suitable finite element spaces. In particular,
$\X_{h}\times V_{h}\times\W_{h}$
could be chosen from any of the known stable triplets for 
linear elasticity with weak stress symmetry, e.g.
\cite{stenberg1988family,Amara-Thomas,ArnAwaQiu,arnold2007mixed,Awanou-rect-weak,cockburn2010new,gopalakrishnan2012second,brezzi2008mixed,BBF-reduced,FarFor,msmfe-quads,lee2016towards}, satisfying the inf-sup condition
\begin{align}
  \forall v\in V_{h}, \xi\in\W_{h}, \quad
  \|v\|+\|\xi\| & \le C\sup_{0\ne\tau\in\mathbb{X}_{h}}\frac{\inp[v]{\dvr{\tau}}+\inp[\xi]{\tau}}{\|\tau\|_{\text{div}}}.\label{eq:inf-sup-elast}
\end{align}
For the flow part, $Z_{h}\times W_{h}$ could be chosen from any
of the known stable velocity-pressure pairs of MFE spaces such as the Raviart-Thomas
($\mathcal{RT}$)
or Brezzi-Douglas-Marini ($\mathcal{BDM}$) spaces, see \cite{brezzi1991mixed}, satisfying
the inf-sup condition
\begin{align}
  \forall w \in W_h, \quad \|w\| & \le
  C\sup_{0\ne q\in Z_{h}}\frac{(\dvr q,w)}{\|q\|_{\text{div}}}\label{eq:inf-sup-darcy}.
\end{align}

\section{Monolithic domain decomposition method}
  \label{subsec:Monolithic-MFE-Domain}

Let $\O=\cup_{i=1}^{m}\Omega_{i}$ be a union of non-overlapping
shape-regular polygonal subdomains, where each subdomain is a union
of elements of $\Tc_h$. Let
$\G_{i,j}=\dO_{i}\cap\dO_{j},\,\G=\cup_{i,j=1}^{m}\G_{i,j},$ and
$\G_{i}=\dO_{i}\cap\G=\dO_{i}\setminus\dO$ denote the interior
subdomain interfaces. Denote the restriction of the spaces $\X_{h}$,
$V_{h}$, $\W_{h}$, $Z_{h}$, and $W_{h}$ to $\O_{i}$ by $\X_{h,i}$,
$V_{h,i}$, $\W_{h,i}$, $Z_{h,i}$, and $W_{h,i}$, respectively. Let
$\Tc_{h,i,j}$ be a finite element partition of $\Gamma_{i,j}$ obtained
from the trace of $\Tc_{h}$, and let $n_{i,j}$ be a unit normal vector
on $\Gamma_{i,j}$ with an arbitrarily fixed direction. In the domain
decomposition formulation we utilize a vector Lagrange
multiplier $\lambda_{h}=(\lambda_{h}^u,\lambda_{h}^p)^{T}$
approximating the displacement and the pressure on the interface and
used to impose weakly the continuity of the normal components of the
poroelastic stress tensor $\s$ and the velocity vector $z$,
respectively. We define the Lagrange multiplier space on $\Tc_{i,j}$
and $\cup_{i<j}\Tc_{i,j}$ as follows:
$$
\Lambda_{h,i,j} := \begin{pmatrix}\Lambda_{h,i,j}^{u}\\
\Lambda_{h,i,j}^{p}
\end{pmatrix} := \begin{pmatrix}\X_{h}\,n_{i,j}\\
Z_{h}\cdot n_{i,j}
\end{pmatrix}, \quad
\Lambda_{h}^{u} := \bigoplus_{1\le i < j\le m}\Lambda_{h,i,j}^{u}, \quad
\Lambda_{h}^{p} := \bigoplus_{1\le i < j\le m}\Lambda_{h,i,j}^{p}, \quad
\Lambda_{h} := \begin{pmatrix}\Lambda_{h}^{u}\\
\Lambda_{h}^{p}
\end{pmatrix}.
$$
The domain decomposition formulation for the mixed Biot problem
in a semi-discrete form reads as follows: for $1\le i\le m$, find
$(\sigma_{h,i},u_{h,i},\g_{h,i},z_{h,i},p_{h,i},\lambda_{h}):[0,T]\to \X_{h,i}\times V_{h,i}\times\W_{h,i}\times Z_{h,i}\times W_{h,i}\times\Lambda_{h}$
such that, for a.e. $t\in (0,T)$,
\begin{align}
  & \inp[A(\sigma_{h,i} + \a p_{h,i}I)]{\tau}_{\O_{i}}
  +\inp[u_{h,i}]{\dvr{\tau}}_{\O_{i}}
  +\inp[\gamma_{h,i}]{\tau}_{\O_{i}} \nonumber \\
  & \qquad\quad =\gnp[g_{u}]{\t\,n_i}_{\dO_i\cap\Gd^{u}} + \gnp[\lambda_{h}^{u}]{\t\,n_i}_{\Gamma_{i}},
  && \forall\tau\in\X_{h,i},\label{eq:dd1-mfe1}\\
 & \inp[\dvr{\sigma_{h,i}}]{v}_{\O_{i}}=-\inp[f]{v}_{\O_{i}}, &  & \forall v\in V_{h,i},\label{eq:dd1-mfe2}\\
 & \inp[\sigma_{h,i}]{\xi}_{\O_{i}}=0, &  & \forall\xi\in\W_{h,i},\label{eq:dd1-mfe3}\\
 & \inp[\K z_{h,i}]{q}_{\O_{i}}-\inp[p_{h,i}]{\dvr{q}}_{\O_{i}}=-\gnp[g_{p}]{q\cdot n_i}_{\dO_i\cap\Gd^p}-\gnp[\lambda_{h}^{p}]{q\cdot n_i}_{\Gamma_{i}}, &  & \forall q\in Z_{h,i},\label{eq:dd1-mfe4}\\
  & c_{0}\inp[\dt{p_{h,i}}]{w}_{\O_{i}}
  +\a\inp[\dt A(\sigma_{h,i} + \a p_{h,i}I)]{wI}_{\O_{i}}
  +\inp[\dvr{z_{h,i}}]{w}_{\O_{i}}=\inp[g]{w}_{\O_{i}}, &  & \forall w\in W_{h,i},\label{eq:dd1-mfe5}\\
 & \sum_{i=1}^{m}\gnp[\sigma_{h,i}\,n_{i}]{\mu^{u}}_{\G_{i}}=0, &  & \forall\mu^{u}\in\Lambda_{h}^{u},\label{eq:dd1-mfe6}\\
 & \sum_{i=1}^{m}\gnp[z_{h,i}\cdot n_{i}]{\mu^{p}}_{\G_{i}}=0, &  & \forall\mu^{p}\in\Lambda_{h}^{p},\label{eq:dd1-mfe7}
\end{align}
where $n_{i}$ is the outward unit normal vector field on $\O_{i}$.
We note that both the elasticity and flow subdomain problems in the above
method are of Dirichlet type. It is easy to check that
\eqref{eq:dd1-mfe1}--\eqref{eq:dd1-mfe7}
is equivalent to the global formulation \eqref{eq:fe1}--\eqref{eq:fe5} with
$(\sigma_{h},u_{h},\gamma_{h},z_{h},p_{h})|_{\Omega_i} = (\sigma_{h,i},u_{h,i},\g_{h,i},z_{h,i},p_{h,i})$.

\subsection{Time discretization}

For time discretization we employ the backward Euler method.
Let $\{t_n\}_{n=0}^N$, $t_n = n \Delta t$, $\Delta t = T/N$, be a uniform partition of
$(0,T)$. The fully discrete problem corresponding to (\ref{eq:dd1-mfe1})--(\ref{eq:dd1-mfe7}) reads as follows: for $0 \le n \le N-1$ and
$1\le i\le m$, find
$(\sigma_{h,i}^{n+1},u_{h,i}^{n+1},\g_{h,i}^{n+1},z_{h,i}^{n+1},p_{h,i}^{n+1},\lambda_{h}^{n+1})
\in\X_{h,i}\times V_{h,i}\times\W_{h,i}\times Z_{h,i}\times W_{h,i}\times\Lambda_{h}$
such that:

\begin{align}
  & \inp[A(\sigma_{h,i}^{n+1} + \a p_{h,i}^{n+1}I)]{\tau}_{\O_{i}}
  +\inp[u_{h,i}^{n+1}]{\dvr{\tau}}_{\O_{i}}+\inp[\gamma_{h,i}^{n+1}]{\tau}_{\O_{i}}\nonumber \\
 & \qquad =\gnp[g_{u}^{n+1}]{\t\,n_i}_{\dO_i\cap\Gd^u}+\gnp[\lambda_{h}^{u,n+1}]{\t\,n_i}_{\Gamma_{i}}, &  & \forall\tau\in\X_{h,i},\label{eq:dd1-mfe1-dsc}\\
 & \inp[\dvr{\sigma_{h,i}^{n+1}}]{v}_{\O_{i}}=-\inp[f^{n+1}]{v}_{\O_{i}}, &  & \forall v\in V_{h,i},\label{eq:dd1-mfe2-dsc}\\
 & \inp[\sigma_{h,i}^{n+1}]{\xi}_{\O_{i}}=0, &  & \forall\xi\in\W_{h,i},\label{eq:dd1-mfe3-dsc}\\
 & \inp[\K z_{h,i}^{n+1}]{q}_{\O_{i}}-\inp[p_{h,i}^{n+1}]{\dvr{q}}_{\O_{i}}=-\gnp[g_{p}^{n+1}]{q\cdot n_i}_{\dO_i\cap\Gd^p}-\gnp[\lambda_{h}^{p,n+1}]{q\cdot n_i}_{\Gamma_{i}}, &  & \forall q\in Z_{h,i},\label{eq:dd1-mfe4-dsc}\\
 & c_{0}\inp[\frac{p_{h,i}^{n+1}-p_{h,i}^{n}}{\Delta t}]{w}_{\O_{i}}+\a\inp[\frac{A\left(\sigma_{h,i}^{n+1}-\sigma_{h,i}^{n}\right)}{\Delta t}]{wI}_{\O_{i}}\nonumber \\
 & \qquad +\a\inp[A\a\frac{p_{h,i}^{n+1}-p_{h,i}^{n}}{\Delta t}I]{w I}_{\O_{i}}+\inp[\dvr{z_{h,i}^{n+1}}]{w}_{\O_{i}}=\inp[g^{n+1}]{w}_{\O_{i}}, &  & \forall w\in W_{h,i},\label{eq:dd1-mfe5-dsc}\\
 & \sum_{i=1}^{m}\gnp[\sigma_{h,i}^{n+1}\,n_{i}]{\mu^{u}}_{\G_{i}}=0, &  & \forall\mu^{u}\in\Lambda_{h}^{u},\label{eq:dd1-mfe6-dsc}\\
 & \sum_{i=1}^{m}\gnp[z_{h,i}^{n+1}\cdot n_{i}]{\mu^{p}}_{\G_{i}}=0, &  & \forall\mu^{p}\in\Lambda_{h}^{p}.\label{eq:dd1-mfe7-dsc}
\end{align}

\begin{remark}\label{rem:init}
We note that the scheme requires initial data $p_{h,i}^0$ and
$\sigma_{h,i}^0$. Such data can be obtained by taking $p_{h,i}^0$ to
be the $L^2$-projection of $p_0$ onto $W_{h,i}$ and solving a mixed elasticity
domain decomposition problem obtained from \eqref{eq:dd1-mfe1-dsc}--\eqref{eq:dd1-mfe3-dsc}
and \eqref{eq:dd1-mfe6-dsc} with $n = -1$.
\end{remark}

\subsection{Time-differentiated elasticity formulation}
In the monolithic domain decomposition method we will utilize a
related formulation in which the first elasticity equation is
differentiated in time. The reason for this will become clear in the
analysis of the resulting interface problem.
We introduce new variables $\dot{u} = \dt u$ and $\dot{\g} = \dt \g$ representing the
time derivatives of the displacement and the rotation, respectively. The time-differentiated
equation \eqref{eq:cts1} is
\begin{equation*}\label{eq:cts1-dt}
\inp[\dt A(\sigma + \a p I)]{\tau} + \inp[\dot{u}]{\dvr{\tau}}
+\inp[\dot{\gamma}]{\tau}=\gnp[\dt g_{u}]{\t\,n}_{\Gamma_{D}^u}, \quad \forall \, \tau \in \X.
\end{equation*}
The semi-discrete equation \eqref{eq:fe1} is replaced by
\begin{equation*}\label{eq:fe1-dt}
\inp[\dt A(\sigma_h + \a p_h I)]{\tau} + \inp[\dot{u}_h]{\dvr{\tau}}
+\inp[\dot{\gamma}_h]{\tau}=\gnp[\dt g_{u}]{\t\,n}_{\Gamma_{D}^u}, \quad \forall \, \tau \in \X_h.
\end{equation*}
We note that the original variables $u_h$ and $\gamma_h$ can be recovered
easily from the solution of the time-differentiated problem. In particular, given
compatible initial data $\sigma_{h,0}$, $u_{h,0}$, $\gamma_{h,0}$ that satisfy
\eqref{eq:fe1}, the expressions
$$
u_h(t) = u_{h,0} + \int_0^t \dot{u}_h(s) \, ds,
\quad \gamma_h(t) = \gamma_{h,0} + \int_0^t \dot{\gamma}_h(s) \, ds,
$$
provide a solution to \eqref{eq:fe1} at any $t \in (0,T]$.

In the domain decomposition formulation we now consider the Lagrange
multiplier $\lambda_h = (\lambda_h^{\dot{u}},\lambda_h^p) \in
\Lambda_h$, where $\lambda_h^{\dot{u}} \in \Lambda_h^u$ approximates
the trace of $\dot{u}$ on $\Gamma$. Then the semi-discrete domain decomposition
equation \eqref{eq:dd1-mfe1} is replaced by
\begin{align*}
  & \inp[\dt A(\sigma_{h,i} + \a p_{h,i}I)]{\tau}_{\O_{i}}
  +\inp[\dot{u}_{h,i}]{\dvr{\tau}}_{\O_{i}}
  +\inp[\dot{\gamma}_{h,i}]{\tau}_{\O_{i}} =\gnp[\dt g_{u}]{\t\,n_i}_{\dO_i\cap\Gd^{u}}
  + \gnp[\lambda_{h}^{\dot{u}}]{\t\,n_i}_{\Gamma_{i}}, \quad
  \forall\tau\in\X_{h,i}. 
\end{align*}
Finally, the fully discrete equation \eqref{eq:dd1-mfe1-dsc} is replaced by
\begin{align}
  & \inp[A(\sigma_{h,i}^{n+1} + \a p_{h,i}^{n+1}I)]{\tau}_{\O_{i}}
  + \Delta t \inp[\dot{u}_{h,i}^{n+1}]{\dvr{\tau}}_{\O_{i}}
  + \Delta t \inp[\dot{\gamma}_{h,i}^{n+1}]{\tau}_{\O_{i}}\nonumber \\
  & \qquad = \Delta t \gnp[\dt g_{u}^{n+1}]{\t\,n_i}_{\dO_i\cap\Gd^u}
  + \Delta t \gnp[\lambda_{h}^{\dot{u},n+1}]{\t\,n_i}_{\Gamma_{i}}
+ \inp[A(\sigma_{h,i}^{n} + \a p_{h,i}^{n}I)]{\tau}_{\O_{i}}
  , &  & \forall\tau\in\X_{h,i}.\label{eq:dd1-mfe1-dsc-dt}
\end{align}

The original variables can be recovered from
\begin{equation}\label{init-recover}
u_h^n = u_h^0 + \Delta t \sum_{k=1}^n \dot{u}_h^k, \quad \gamma_h^n
= \gamma_h^0 + \Delta t  \sum_{k=1}^n \dot{\gamma}_h^k, \quad
\lambda_h^{u,n} = \lambda_h^{u,0} + \Delta t \sum_{k=1}^n \dot \lambda_h^{u,k}.
\end{equation}

\subsection{Reduction to an interface problem}

The non-overlapping domain decomposition algorithm for the solution of
\eqref{eq:dd1-mfe1-dsc-dt},
\eqref{eq:dd1-mfe2-dsc}--\eqref{eq:dd1-mfe7-dsc} at each time step is based on
reducing it to an interface problem for the Lagrange multiplier $\lambda_h$.
To this end, we introduce two sets of complementary subdomain problems.
The first set of problems reads as follows: for $1\le i\le m$, find
$(\sb,\ubd,\gbd,\zb,\pb)\in\X_{h,i}\times V_{h,i}\times\W_{h,i}\times Z_{h,i}\times W_{h,i}$
such that
\begin{align}
  & \inp[A(\sb + \a\pb I)]{\tau}_{\O_{i}}
  + \Delta t \inp[\ubd]{\dvr{\tau}}_{\O_{i}} + \Delta t \inp[\gbd]{\tau}_{\O_{i}}
  \nonumber \\
  & \quad\qquad = \Delta t \gnp[\dt g_{u}^{n+1}]{\t\,n_i}_{\dO_i\cap\Gd^u}
  + \inp[A(\sigma_{h,i}^{n} + \a p_{h,i}^{n} I)]{\tau}_{\O_{i}},
  &  & \forall\tau\in\X_{h,i},\label{eq:dd1-mfe1-bar}\\
 & \inp[\dvr{\sb}]{v}_{\O_{i}}=-\inp[f^{n+1}]{v}_{\O_{i}}, &  & \forall v\in V_{h,i},\label{eq:dd1-mfe2-bar}\\
 & \inp[\sb]{\xi}_{\O_{i}}=0, &  & \forall\xi\in\W_{h,i},\label{eq:dd1-mfe3-bar}\\
 & \inp[\K\zb]{q}_{\O_{i}}-\inp[\pb]{\dvr{q}}_{\O_{i}}=-\gnp[g_{p}^{n+1}]{q\cdot n_i}_{\dO_i\cap\Gd^p}, &  & \forall q\in Z_{h,i},\label{eq:dd1-mfe4-bar}\\
 & c_{0}\inp[\pb]{w}_{\O_{i}}+\a\inp[A(\sb + \a\pb I)]{wI}_{\O_{i}}
  +\Delta t\inp[\dvr{\zb}]{w}_{\O_{i}}\nonumber \\
  & \quad\qquad = \Delta t\inp[g^{n+1}]{w}_{\O_{i}}+c_{0}\inp[p_{h,i}^{n}]{w}_{\O_{i}}
  +\a\inp[A(\sigma_{h,i}^{n} + \a p_{h,i}^{n} I)]{wI}_{\O_{i}}, &  & \forall w\in W_{h,i}.\label{eq:dd1-mfe5-bar}
\end{align}
These subdomain problems have zero Dirichlet data on the interfaces and
incorporate the true source terms $f$ and $g$
and outside boundary conditions $g_u$ and $g_p$, as well as initial data
$\sigma_{h,i}^n$ and $p_{h,i}^n$.

The second problem set reads as follows: given $\lambda_h \in
\Lambda_h$, for $1\le i\le m$, find
$((\ss(\lambda_h),\usd(\lambda_h), \linebreak
\gsd(\lambda_h),\zs(\lambda_h),\ps(\lambda_h))\in\X_{h,i}\times
V_{h,i}\times\W_{h,i}\times Z_{h,i}\times W_{h,i}$ such that:
\begin{align}
  & \inp[A\big(\ss(\lambda_h) + \a\ps(\lambda_h)I \big) ]{\tau}_{\O_{i}}
  + \Delta t \inp[\usd(\lambda_h)]{\dvr{\tau}}_{\O_{i}}\nonumber \\
  & \qquad\quad + \Delta t \inp[\gsd(\lambda_h)]{\tau}_{\O_{i}}
  = \Delta t \gnp[\lambda_{h}^{\dot{u}}]{\t\,n_i}_{\Gamma_{i}},
  &  & \forall\tau\in\X_{h,i},\label{eq:dd1-mfe1-star}\\
 & \inp[\dvr{\ss(\lambda_h)}]{v}_{\O_{i}}=0, &  & \forall v\in V_{h,i},\label{eq:dd1-mfe2-star}\\
 & \inp[\ss(\lambda_h)]{\xi}_{\O_{i}}=0, &  & \forall\xi\in\W_{h,i},\label{eq:dd1-mfe3-star}\\
  & \inp[\K\zs(\lambda_h)]{q}_{\O_{i}}-\inp[\ps(\lambda_h)]{\dvr{q}}_{\O_{i}}
  =-\gnp[\lambda_{h}^{p}]{q\cdot n_i}_{\Gamma_{i}},
  &  & \forall q\in Z_{h,i},\label{eq:dd1-mfe4-star}\\
  & c_{0}\inp[\ps(\lambda_h)]{w}_{\O_{i}}
  +\a\inp[A\big(\ss(\lambda_h) + \a\ps(\lambda_h)I\big)]{wI}_{\O_{i}}\nonumber \\
  & \qquad\quad + \Delta t \inp[\dvr{\zs(\lambda_h)}]{w}_{\O_{i}}=0,
  &  & \forall w\in W_{h,i}.\label{eq:dd1-mfe5-star}
\end{align}
These problems have $\lambda_h$ as Dirichlet interface data, along with zero source terms,
zero outside boundary conditions, and zero data from the previous time step.

Define the bilinear forms $a_{i}^{n+1}:\Lambda_{h}\times\Lambda_{h}\to\R$,
$1\le i \le m$, $a^{n+1}:\Lambda_{h}\times\Lambda_{h}\to\R$, and
the linear functional $g^{n+1}:\Lambda_{h}\to\R$ for all $0\le n\le N-1$
by 
\begin{align}
  a_{i}^{n+1}(\lambda_{h},\mu) & = \gnp[\ss(\lambda_{h})\,n_{i}]{\mu^{u}}_{\Gamma_i}
  - \gnp[\zs(\lambda_{h})\cdot n_{i}]{\mu^{p}}_{\Gamma_i},
  \quad a^{n+1}(\lambda_{h},\mu) =\sum_{i=1}^{m}a_{i}^{n+1}(\lambda_{h},\mu),
  \label{eq:interface-operator}\\
  g^{n+1}(\mu) & = \sum_{i=1}^{m}\left(-\gnp[\sb\,n_{i}]{\m^{u}}_{\Gamma_i}
  + \gnp[\zb\cdot n_{i}]{\m^{p}}_{\Gamma_i} \right).
\end{align}
It follows from (\ref{eq:dd1-mfe6-dsc})--(\ref{eq:dd1-mfe7-dsc})
that, for each $0\le n\le N-1$, the solution to the global problem
\eqref{eq:dd1-mfe1-dsc-dt},
\eqref{eq:dd1-mfe2-dsc}--\eqref{eq:dd1-mfe7-dsc} 
is equivalent to solving
the interface problem for $\lambda_h^{n+1} \in \Lambda_h$:
\begin{align}
a^{n+1}(\lambda_{h}^{n+1},\mu) = g^{n+1}(\m),\quad\forall\m\in\Lambda_{h},\label{eq:interface-prob}
\end{align}
and setting
\begin{align*}
  \begin{aligned} & \s_{h,i}^{n+1}=\ss(\lambda_{h}^{n+1})+\sb,
    &  & \dot{u}_{h,i}^{n+1}=\usd(\lambda_{h}^{n+1})+\ubd,
    && \dot{\g}_{h,i}^{n+1} = \gsd(\lambda_{h}^{n+1})+\gbd,\\
 & z_{h,i}^{n+1}=\zs(\lambda_{h}^{n+1})+\zb, &  & p_{h,i}^{n+1}=\ps(\lambda_{h}^{n+1})+\pb.
\end{aligned}
\end{align*}

\subsection{Analysis of the interface problem}

We next show that the interface bilinear form $a^{n+1}(\cdot,\cdot)$
is positive definite, which implies that the interface problem
\eqref{eq:interface-prob} is well-posed and can be solved using a
suitable Krylov space method such as GMRES. We further obtain bounds
on the spectrum of $a^{n+1}(\cdot,\cdot)$ and establish rate of
convergence for GMRES. We start by obtaining an expression for
$a^{n+1}(\cdot,\cdot)$ in terms of the subdomain bilinear forms.

\begin{prop}
\label{prop:interface-formula} For $\lambda_{h},\m\in\Lambda_{h}$,
the interface bilinear form can be expressed as 
\begin{align}
  a^{n+1}(\lambda_{h},\m) & =
\frac{1}{\Delta t}
  \sum_{i=1}^{m}\bigg[\inp[A\ss(\m)]{\ss(\lambda_{h})}+2\inp[A\alpha\ps(\m)I]{\ss(\lambda_{h})}_{\O_{i}}\nonumber \\
  & +\inp[A\a\ps(\mu)I]{\a\ps(\lambda_{h})I}_{\O_{i}}
  +c_{0}\inp[\ps(\mu)]{\ps(\lambda_{h})}_{\O_{i}}\nonumber \\
 & + \Delta t \inp[K^{-1}\zs(\mu)]{\zs(\lambda_{h})}\bigg].\label{eq:interface-formula}
\end{align}
\end{prop}

\begin{proof}
To see this, consider the second set of complementary equations  \eqref{eq:dd1-mfe1-star}--\eqref{eq:dd1-mfe5-star}
with data $\mu$, use the test functions: $\ss(\lambda_{h})$ in (\ref{eq:dd1-mfe1-star})
and $\zs(\lambda_{h})$ in equation (\ref{eq:dd1-mfe4-star}) . 
\end{proof}

\begin{remark}
  The non-differentiated formulation results in a missing scaling of $\frac{1}{\Delta t}$
  in the term \linebreak
  $\inp[A\big(\ss(\lambda_h) + \a\ps(\lambda_h)I \big) ]{\tau}_{\O_{i}}$ in
  \eqref{eq:dd1-mfe1-star} compared to the similar term in \eqref{eq:dd1-mfe5-star}. Therefore
  the two terms cannot be combined, resulting in a non-coercive expression for
  $a^{n+1}(\cdot,\cdot)$.  
  \end{remark}

Recalling the properties of $A$ and $K$, there exist constants
  $0 < a_{\min} \le a_{\max} < \infty$ and $0 < k_{\min} \le k_{\max} < \infty$ such that
\begin{align}
  a_{\min}\|\t\|^{2}\le\inp[A\t]{\t}\le a_{\max}\|\t\|^{2}, \quad
\forall \, \t \in \X,\label{eq:coercivity-elast}\\
k_{\min}\|q\|^{2}\le\inp[K q]{q}\le k_{\max}\|q\|^{2}, \quad \forall \, q \in Z.
\label{eq:coercivity-flow}
\end{align}
We will also utilize suitable mixed interpolants in the finite element spaces $\X_{h,i}$ and
$Z_{h,i}$. It is shown in \cite{eldar_elastdd} that there exists an interpolant
$\Pit_i: H^{\epsilon}(\O_{i},\M)\cap\X_{i} \to \X_{h,i}$ for any $\epsilon > 0$ such that
for all $\sigma \in H^{\epsilon}(\O_{i},\M)\cap\X_{i}$, $\tau \in \X_{h,i}$, $v \in V_{h,i}$,
and $\xi \in \W_{h,i}$,
\begin{equation}\label{Pi-tilde-prop}
  (\dvr (\Pit_i \sigma - \sigma), v)_{\Omega_i} = 0, \quad
  (\Pit_i \sigma - \sigma, \xi)_{\Omega_i} = 0, \quad
  \langle(\Pit_i \sigma - \sigma)n_i,\tau \, n_i\rangle_{\partial \Omega_i} = 0,
\end{equation}
and
\begin{equation}\label{Pi-tilde-bound}
\|\Pit_i \sigma\|_{\Omega_i} \le C (\|\sigma\|_{\epsilon,\Omega_i} + \|\dvr \sigma\|_{\Omega_i}).
\end{equation}
For the Darcy problem we use the canonical mixed interpolant \cite{brezzi1991mixed},
$\Pi: H^{\epsilon}(\O_{i},\R^d)\cap Z_{i} \to Z_{h,i}$ such that for all
$z \in H^{\epsilon}(\O_{i},\R^d)\cap Z_{i}$, $q \in Z_{h,i}$, and $w \in W_{h,i}$,
\begin{equation}\label{Pi-prop}
  (\dvr (\Pi_i z - z), w)_{\Omega_i} = 0, \quad
  \langle(\Pi_i z - z)\cdot n_i,q \cdot n_i\rangle_{\partial \Omega_i} = 0,
\end{equation}
and
\begin{equation}\label{Pi-bound}
\|\Pi_i z\|_{\Omega_i} \le C (\|z\|_{\epsilon,\Omega_i} + \|\dvr z\|_{\Omega_i}).
\end{equation}

\begin{lem}
The interface bilinear form $a^{n+1}(\cdot,\cdot)$ is positive
definite over $\Lambda_{h}$. 
\end{lem}

\begin{proof}
Using the representation of the interface bilinear form (\ref{eq:interface-formula}),
we get
\begin{multline}
  a^{n+1}(\lah,\lah) = \frac{1}{\Delta t} \sum_{i=1}^m \bigg[\inp[A(\ss(\lambda_{h})
      +\alpha\ps(\lambda_{h})I)]{\ss(\lambda_{h})
      +\alpha\ps(\lambda_{h})I}_{\O_{i}}\\
    +c_{0}\inp[\ps(\lambda_{h})]{\ps(\lambda_{h})}_{\O_{i}}\
    + \Delta t \inp[\K\zs(\lambda_{h})]{\zs(\lambda_{h})}_{\O_{i}}\bigg],
  \label{eq:interface-coer-sym}
\end{multline}
which, combined with \eqref{eq:coercivity-elast}--\eqref{eq:coercivity-flow}, gives
$a^{n+1}(\lah,\lah)\ge0$, and hence $a^{n+1}(\cdot,\cdot)$ is positive semidefinite.
We next show that $a(\lambda_{h},\lambda_{h})=0$ implies $\lambda_{h}=0$. We use a two-part
argument to control separately $\lambda_{h}^{\dot{u}}$ and $\lambda_h^p$.  
Let $\O_{i}$ be a domain adjacent to $\Gamma_{D}^u$ such that $|\dO_{i}\cap\G_{D}^u| > 0$
and let $(\psi^{\ud},\phi^{\ud})$
be the solution of the auxiliary elasticity problem
\begin{align*}
  & A\psi_{i}^{\ud}=\epsilon(\phi_{i}^{\ud}),\quad\dvr\psi_{i}^{\ud}=0\quad\text{in }\O_{i}, \\
 & \phi_{i}^{\ud}=0\quad\text{on }\dO_{i}\cap\G_{D}^u,\\
 & \psi_{i}^{\ud}\,n_{i}=\begin{cases}
0 & \mbox{on }\dO_{i}\cap\G_{N}^\sigma\\
\lambda_{h}^{\ud} & \mbox{on }\G_{i}.
\end{cases}
\end{align*}
Elliptic regularity \cite{Dauge} implies that
$\psi_{i}^{\ud}\in H^{\epsilon}(\O_{i},\M)\cap\X_{i}$ for some
$\epsilon > 0$, and therefore the mixed interpolant $\Pit_{i}\psi_{i}^{\ud}$
is well defined. Taking $\tau=\Pit_{i}\psi_{i}^{\ud}$ in \eqref{eq:dd1-mfe1-star}
and using \eqref{Pi-tilde-prop} and \eqref{Pi-tilde-bound} gives
\begin{align}
  \|\lambda_{h}^{\ud}\|_{\G_{i}}^{2} & =
  \gnp[\lambda_{h}^{\ud}]{\psi_{i}^{\ud}\,n_{i}}_{\G_{i}}
  = \gnp[\lambda_{h}^{\ud}]{\Pit\psi_{i}^{\ud}\,n_{i}}_{\G_{i}} \nonumber \\
  & = \frac{1}{\Delta t}\inp[A(\ss(\lambda_{h})
    + \a\ps(\lambda_{h})I)]{\Pit\psi_{i}^{\ud}}_{\O_{i}}
  + \inp[\usd(\lambda_h)]{\dvr{\Pit\psi_{i}^{\ud}}}_{\O_{i}}
  + \inp[\gsd(\lambda_h)]{\Pit\psi_{i}^{\ud}}_{\O_{i}} \nonumber \\
    & = \frac{1}{\Delta t}\inp[A^{1/2}(\ss(\lambda_{h})
    + \a\ps(\lambda_{h})I)]{A^{1/2}\Pit\psi_{i}^{\ud}}_{\O_{i}} \nonumber \\
  & \le \frac{C}{\Delta t} \|A^{1/2}(\ss(\lambda_{h})+\alpha\ps(\lambda_{h})I)\|_{\O_{i}}
  \|\psi_{i}^{\ud}\|_{\epsilon,\O_{i}} \nonumber \\
  & \le \frac{C}{\Delta t}\|A^{1/2}(\ss(\lambda_{h})+\alpha\ps(\lambda_{h})I)\|_{\O_{i}}
  \|\lambda_{h}^{\ud}\|_{\G_{i}}, \label{aux-bound-lambda-u}
\end{align}
where in the last inequality we used the elliptic regularity bound \cite{Dauge}
\begin{equation}\label{ell-reg}
  \|\psi_{i}^{\ud}\|_{\epsilon,\O_{i}} \le C \|\lambda_h\|_{\epsilon-1/2,\Gamma_i}.
\end{equation}
Using the
representation of the interface bilinear form \eqref{eq:interface-formula}, we obtain
\begin{equation}
  \|\lambda_{h}^{\ud}\|_{\G_{i}}^{2}\le \frac{C}{\Delta t} a_{i}^{n+1}(\lambda_{h},\lambda_{h})\,\,\,\,\,
  \forall \, \lambda_{h}\in\Lambda_{h}.\label{eq:coerc-ineq-sub}
\end{equation}
Next, consider an adjacent subdomain $\Omega_j$ such that $|\Gamma_{ij}| > 0$.
Let $(\psi_j^{\ud},\phi_j^{\ud})$ be the solution to
\begin{align*}
  & A\psi_{j}^{\ud}=\epsilon(\phi_{j}^{\ud}),\quad\dvr\psi_{j}^{\ud}=0\quad\text{in }\O_{j},\\
 & \phi_{i}^{\ud}=0\quad\text{on } \Gamma_{ij},\\
 & \psi_{i}^{\ud}\,n_{i}=\begin{cases}
0 & \mbox{on }\dO_{j}\cap\dO\\
\lambda_{h}^{\ud} & \mbox{on }\G_{j}\setminus\Gamma_{ij}.
\end{cases}
\end{align*}
Taking $\tau=\Pit_j\psi_{j}^{\ud}$ in \eqref{eq:dd1-mfe1-star}
and using \eqref{Pi-tilde-prop} gives
\begin{align*}
  \|\lambda_{h}^{\ud}\|_{\G_{j}\setminus\G_{ij}}^{2} & = \frac{1}{\Delta t}\inp[A(\ssj(\lambda_{h})
    + \a\psj(\lambda_{h})I)]{\Pit\psi_{j}^{\ud}}_{\O_{j}}
  - \langle \lambda_{h}^{\ud},\psi_{j}^{\ud}\, n_j \rangle_{\Gamma_{ij}} \\
  & \le C\left(\frac{1}{\Delta t}\|A^{1/2}(\ssj(\lambda_{h}) + \a\psj(\lambda_{h})I)\|_{\Omega_j}
  + \|\lambda_{h}^{\ud}\|_{\Gamma_{ij}}\right)\|\psi_{j}^{\ud}\|_{\epsilon,\Omega_j} \\
  & \le \frac{C}{\sqrt{\Delta t}}\left(a_{j}^{n+1}(\lambda_{h},\lambda_{h})^{1/2}
  + a_{i}^{n+1}(\lambda_{h},\lambda_{h})^{1/2}\right)
  \|\lambda_{h}^{\ud}\|_{\G_{j}\setminus\G_{ij}},
\end{align*}
where in the first inequality we used \eqref{Pi-tilde-bound} and the trace inequality
\cite{Mathew}
\begin{equation*}
  \langle \tau \, n_j, \mu \rangle_{\Gamma_{ij}} \le C (\|\tau\|_{\epsilon,\Omega_j}
  + \|\dvr \tau\|_{\Omega_j})\|\mu\|_{\Gamma_{ij}}, \quad \forall \,
  \tau \in H^{\epsilon}(\O_{j},\M)\cap\X_{j}, \, \mu \in L^2(\Gamma_{ij},\R^d),
\end{equation*}
and for the second inequality we used the representation \eqref{eq:interface-formula}
and the bound from $\Omega_i$ \eqref{eq:coerc-ineq-sub}, along with the elliptic regularity
bound \eqref{ell-reg}. Iterating over all subdomains in a similar fashion results in
\begin{equation}\label{bound-lambda-u}
  \|\lambda_{h}^{\ud}\|_{\G}^{2}\le \frac{C}{\Delta t} a^{n+1}(\lambda_{h},\lambda_{h})\,\,\,\,\,
  \forall \, \lambda_{h}\in\Lambda_{h}.
\end{equation}

The argument for $\lambda_h^p$ is similar. We start with a subdomain $\Omega_i$
adjacent to $\Gamma_{D}^p$ such that $|\dO_{i}\cap\G_{D}^p| > 0$
and let $(\psi^{p},\phi^{p})$ be the solution of the auxiliary flow problem
\begin{align}
  & K^{-1}\psi_{i}^{p}=\nabla \phi_{i}^{p}, \quad \nabla\cdot\psi_{i}^{p}=0\quad\text{in }\O_{i},
  \label{aux2-1}\\
 & \phi_{i}^{p}=0\quad\text{on }\dO_{i}\cap\G_{D}^p,\\
 & \psi_{i}^{p}\cdot n_{i}=\begin{cases}
0 & \mbox{on }\dO_{i}\cap\G_{N}^z,\\
\lambda_{h}^{p} & \mbox{on }\G_{i}.
\end{cases}\label{aux2-2}
\end{align}
Taking $q=\Pi_{i}\psi_{i}^{p}$ in \eqref{eq:dd1-mfe4-star} and using
\eqref{Pi-prop}, \eqref{Pi-bound}, and elliptic regularity similar to \eqref{ell-reg} gives
\begin{align*}
  \|\lambda_h^p\|_{\Gamma_i}^2 & = \<\lambda_h^p,\psi_{i}^{p}\cdot n_i\>_{\G_i}
  = \<\lambda_h^p,\Pi_i\psi_{i}^{p}\cdot n_i\>_{\G_i}
  = (\K\zs(\lambda_{h}),\Pi_i\psi_{i}^{p})_{\Omega_i} \\
  & \le C \|K^{-1/2}\zs(\lambda_{h})\|_{\Omega_i}\|\psi_i^p\|_{\epsilon,\O_i}
  \le C \|K^{-1/2}\zs(\lambda_{h})\|_{\Omega_i}\|\lambda_h^p\|_{\Gamma_i},
\end{align*}
which, together with \eqref{eq:interface-formula} implies
$$
\|\lambda_h^p\|_{\Gamma_i}^2 \le C a_{i}^{n+1}(\lambda_{h},\lambda_{h}).
$$
Iterating over all subdomains in a way similar to the argument for $\lambda_h^{\ud}$,
we obtain
\begin{equation}\label{bound-lambda-p}
  \|\lambda_{h}^p\|_{\G}^{2}\le C a^{n+1}(\lambda_{h},\lambda_{h})\,\,\,\,\,
  \forall \, \lambda_{h}\in\Lambda_{h}.
\end{equation}
A combination of \eqref{bound-lambda-u} and \eqref{bound-lambda-p} implies that
$a^{n+1}(\cdot,\cdot)$ is positive definite on $\Lambda_h$.
 \end{proof}

\begin{thm} \label{thm:intr-op-bound}
There exist positive constants $C_{0}$ and $C_{1}$ independent of
$h$ and $\Delta t$ such that
\begin{align}
  \forall \, \lambda_{h}\in\Lambda_{h},\qquad
  C_{0}( \Delta t \|\lambda_{h}^{\ud}\|_{\G}^{2} + \|\lambda_{h}^p\|_{\G}^{2})
  \le a^{n+1}(\lambda_{h},\lambda_{h})\le C_{1}h^{-1}
  (\Delta t \|\lambda_{h}^{\ud}\|_{\G}^{2} + \|\lambda_h^p\|_{\G}^2).\label{eq:eigenvalue-bound}
\end{align}
In addition, there exist positive constants $\tilde C_{0}$ and $\tilde C_{1}$ independent of
$h$, $\Delta t$, and $c_0$ such that
\begin{align}
  \forall\lambda_{h}\in\Lambda_{h},\qquad
  \tilde C_{0}( \Delta t \|\lambda_{h}^{\ud}\|_{\G}^{2} + \|\lambda_{h}^p\|_{\G}^{2})
  \le a^{n+1}(\lambda_{h},\lambda_{h})\le \tilde C_{1}h^{-1}\Delta t^{-1} 
  (\Delta t \|\lambda_{h}^{\ud}\|_{\G}^{2} + \|\lambda_h^p\|_{\G}^2).\label{eq:eigenvalue-bound-2}
\end{align}
\end{thm}

\begin{proof}
  The left inequality in \eqref{eq:eigenvalue-bound} and \eqref{eq:eigenvalue-bound-2}
  follows from \eqref{bound-lambda-u} and \eqref{bound-lambda-p}. To prove
  the right inequality, we use the definition of the interface operator
  (\ref{eq:interface-operator}) to obtain
\begin{align}
  a_{i}^{n+1} & (\lambda_{h},\lambda_{h})
  =\gnp[\ss(\lambda_{h})\,n_{i}]{\lambda_{h}^{\ud}}_{\G_{i}}
  -\gnp[\zs(\lambda_{h})\cdot n_{i}]{\lambda_{h}^{p}}_{\G_{i}}\nonumber \\
  & \le\|\ss(\lambda_{h})\,n_{i}\|_{\G_{i}}\|\lambda_{h}^{\ud}\|_{\G_{i}}
  +\|\zs(\lambda_{h})\cdot n_{i}\|_{\G_{i}} \|\lambda_h^p\|_{\G_i}  \nonumber \\
  & \le Ch^{-1/2}\left(\|\ss(\lambda_{h})\|_{\O_{i}}\|\lambda_{h}^{\ud}\|_{\G_{i}}
  +\|\zs(\lambda_{h})\|_{\O_{i}}\|\lambda_{h}^p\|_{\G_{i}}\right)\nonumber \\
  & \le Ch^{-1/2}\left( \big(\|\ss(\lambda_{h})+\alpha\ps(\lambda_{h})I\|_{\O_{i}}
  +\|\alpha\ps(\lambda_{h})I\| \big)\|\lambda_{h}^{\ud}\|_{\G_{i}}
  +\|\zs(\lambda_{h})\|_{\O_{i}}\|\lambda_{h}^p\|_{\G_{i}}\right) \nonumber \\
  & \le C h^{-1/2} a_{i}^{n+1}(\lambda_{h},\lambda_{h})^{1/2}
  \left(\Delta t^{1/2}\|\lambda_{h}^{\ud}\|_{\G_{i}} + \|\lambda_{h}^p\|_{\G_{i}}\right),
  \label{eq:cond-bound-1}
\end{align}
where for the second inequality we used the discrete trace inequality
for finite element functions $\varphi$,
\begin{equation}\label{trace}
  \|\varphi\|_{\Gamma_i} \le C h^{-1/2} \|\varphi\|_{\O_i},
\end{equation}
and the last inequality follows from \eqref{eq:interface-coer-sym}. We note that
the constant in the last inequality depends on $c_0$. This implies
the right inequality in \eqref{eq:eigenvalue-bound}.

To obtain the right inequality in \eqref{eq:eigenvalue-bound-2} with a
constant independent of $c_0$, we use the inf-sup condition
(\ref{eq:inf-sup-darcy}) and (\ref{eq:dd1-mfe4-star}):
\begin{align}
  \|\ps(\lambda_{h})\|_{\Omega_{i}} &
  \le C\sup_{0\ne q\in Z_{h,i}}\frac{\langle\dvr q,\ps(\lambda_{h})\rangle_{\Omega_{i}}}
      {\|q\|_{\dvr,\Omega_i}}
      = C\sup_{0\ne q\in Z_{h,i}}\frac{(K^{-1}\zs(\lah),q)_{\Omega_{i}}
        + \langle\lah^p,q\cdot n_i\rangle_{\Gamma_{i}}}{\|q\|_{\dvr,\Omega_i}}
        \nonumber \\
        & \le C\left(\|\zs(\lambda_h)\|_{\Omega_{i}}
        +h^{-1/2}\|\lah^p\|_{\G_{i}}\right),\label{eq:pres-vel-ineq}
\end{align}
where the last inequality uses \eqref{trace}.
Combining (\ref{eq:pres-vel-ineq}) with the next to last inequality in
(\ref{eq:cond-bound-1})
and using (\ref{eq:interface-coer-sym}), we get: 
\begin{align*}
  a_{i}^{n+1} & (\lambda_{h},\lambda_{h}) \nonumber \\
  & \le C h^{-1/2} \left( \big(\Delta t^{1/2} a_{i}(\lambda_{h},\lambda_{h})^{1/2} 
  + a_{i}(\lambda_{h},\lambda_{h})^{1/2} + h^{-1/2}\|\lambda_h^p\|_{\G_i}\big)\|\lambda_h^{\ud}\|_{\G_i}
  + a_{i}(\lambda_{h},\lambda_{h})^{1/2}\|\lambda_h^p\|_{\G_i} \right)\nonumber \\
  & \le C\left(\epsilon a_{i}^{n+1}(\lambda_{h},\lambda_{h})
  + \frac{1}{\epsilon}h^{-1}\Delta t^{-1}(\Delta t \|\lambda_h^{\ud}\|_{\G_i}^2
  + \|\lambda_h^p\|_{\G_i}) \right),
\end{align*}
using Young's inequality in the last inequality. Taking $\epsilon$ sufficiently small
implies the right inequality in \eqref{eq:eigenvalue-bound-2}.
\end{proof}

Theorem \ref{thm:intr-op-bound} provides upper and lower bounds on the
field of values of the interface operator, which can be used to
estimate the convergence of the interface GMRES solver. In particular,
let $\r_k = (\r_k^{\ud},\r_k^p)$ be the $k$-th residual of the GMRES
iteration for solving the interface problem (\ref{eq:interface-prob}).
Define $|\r_k|_\star^2 = \Delta t |\r_k^{\ud}|^2 + |\r_k^p|^2$, where
$|\cdot|$ denotes the Euclidean vector norm. The following corollary
to Theorem \ref{thm:intr-op-bound} follows from the field-of-values
analysis in \cite{StarkeFOV:97}.

\begin{cor} \label{cor:gmres-bound}
  For the $k$-th GMRES residual for solving (\ref{eq:interface-prob}), it holds that
\begin{equation}
\label{eq:fovbound}
|\r_k|_\star \leq  \left(\sqrt{1- (C_0/C_1)^2 h^2}\right)^k\, |\r_0 |_\star 
\end{equation}
and
\begin{equation}
\label{eq:fovbound-tilde}
|\r_k|_\star \leq \left(\sqrt{1- (\tilde C_0/\tilde C_1)^2 h^2 \Delta t ^2}\right)^k\, |\r_0 |_\star.
\end{equation}
\end{cor}

\begin{remark}
  Bounds \eqref{eq:fovbound} and \eqref{eq:fovbound-tilde} imply
  convergence of the interface GMRES iteration that is independent of
  either $\Delta t$ or $c_0$, but not both. In Section~\ref{sec:numer}
  we present numerical results showing that the GMRES convergence is
  robust with respect to both $c_0$ and $\Delta t$.
\end{remark}

\section{Split methods}\label{sec:split}

In this section, we consider two popular splitting methods to decouple
the fully coupled poroelastic problem, namely the drained split (DS)
and fixed stress (FS) methods \cite{KimDS2011,KimFS2011}. We show,
using energy bounds, that these two methods are unconditionally stable
in our MFE formulation. We then define, at each time step, a domain
decomposition algorithm for the flow and mechanics equations
separately. Domain decomposition techniques for the flow
\cite{glowinski1988domain} and mechanics \cite{eldar_elastdd}
components have already been studied in previous works.

\subsection{Drained split}\label{subsec:Drained-Split}

The DS method consists of solving the mechanics problem first, with
the value of pressure from the previous time step.  Afterward,
the flow problem is solved using the new values of the stress
tensor. The DS method for the classical Biot formulation of
poroelasticity is known to require certain conditions on the
parameters for stability \cite{KimDS2011}. In the setting of our mixed
formulation, we show that this is not necessary and the method is
unconditionally stable, see also \cite{YiSplit}.  For simplicity, we
do the analysis with zero source terms.

The DS method results in the problem: for $n=-1,0,\ldots,N-1$,
find $(\sigma_{h}^{n+1},u_{h}^{n+1},\g_{h}^{n+1},z_{h}^{n+1}, p_{h}^{n+1})\in\X_{h}\times V_{h}\times\W_{h}\times Z_{h}\times W_{h}$ such that
\begin{align}
 & \inp[A\sigma_{h}^{n+1}]{\tau}+\inp[u_{h}^{n+1}]{\dvr{\tau}}+\inp[\gamma_{h}^{n+1}]{\tau}=-\inp[A\a p_{h}^{n}I]{\tau}, & \forall\tau\in\X_{h},\label{eq:ds1}\\
 & \inp[\dvr{\sigma_{h}^{n+1}}]{v}=0, & \forall v\in V_{h},\label{eq:ds2}\\
 & \inp[\sigma_{h}^{n+1}]{\xi}=0, & \forall\xi\in\W_{h},\label{eq:ds3}
\end{align}
and
\begin{align}
 & \inp[\K z_{h}^{n+1}]{q}-\inp[p_{h}^{n+1}]{\dvr{q}}=0, & \forall q\in Z_{h},\label{eq:ds4}\\
  & c_{0}\inp[\frac{p_{h}^{n+1}-p_{h}^{n}}{\Dt}]{w}
  +\a\inp[A\a\frac{p_{h}^{n+1}-p_{h}^{n}}{\Dt}I]{w I}+\inp[\dvr{z_{h}^{n+1}}]{w}=-\a\inp[A\frac{\sigma_{h}^{n+1}-\sigma_{h}^{n}}{\Dt}]{wI}, & \forall w\in W_{h},\label{eq:ds5}
\end{align}
where \eqref{eq:ds1}--\eqref{eq:ds4} hold for $n=-1,0,\ldots,N-1$ with $p_{h}^{-1} := p_{h}^{0}$,
and \eqref{eq:ds5} holds for $n=0,\ldots,N-1$. We note that solving
\eqref{eq:ds1}--\eqref{eq:ds4} for $n=-1$ provides initial data $\sigma_h^0$, $u_h^0$, $\g_h^0$,
and $z_h^0$. 

\subsubsection{Stability analysis for drained split}
The following theorem shows that the drained split scheme is unconditionally stable.

\begin{thm}\label{thm:ds}
For the solution
  $(\sigma_{h}^{n+1},u_{h}^{n+1},\g_{h}^{n+1},z_{h}^{n+1},p_{h}^{n+1})_{0\le n\le N-1}$
of the system  (\ref{eq:ds1})--(\ref{eq:ds5}), there exists a constant $C$ independent of
$h$, $\Delta t$, $c_0$, and $a_{\min}$ such that
\begin{align*}
&  \sum_{n=0}^{N-1}\frac{c_{0}}{\Delta t}\|p_{h}^{n+1}-p_{h}^{n}\|^{2}
  + \max_{0 \le n \le N-1}\left(\|z_{h}^{n+1}\|^2 + \|p_h^{n+1}\|^2 + \|A^{1/2}\sigma_h^{n+1}\|^2
  + \|u_{h}^{n+1}\|^2 +\|\gamma_{h}^{n+1}\|^2 \right) \\
& \qquad\qquad \le C \left(\|p_h^0\|^2 + \|z_h^0\|^2 \right).
\end{align*}
\end{thm}
\begin{proof}
We subtract two successive time steps for equations \eqref{eq:ds1}--\eqref{eq:ds4}, obtaining,
for $n=0,\ldots,N-1$,
\begin{align}
  & \inp[A(\sigma_{h}^{n+1}-\sigma_{h}^{n})]{\tau}+\inp[u_{h}^{n+1}-u_{h}^{n}]{\dvr{\tau}}
  +\inp[\gamma_{h}^{n+1}-\gamma_{h}^{n}]{\tau}=-\inp[A\a(p_{h}^{n}-p_{h}^{n-1})I]{\tau},
  & \forall \, \tau\in\X_{h},\label{eq:ds1-1}\\
  & \inp[\dvr({\sigma_{h}^{n+1}-\sigma_{h}^{n})}]{v}=0,
  & \forall \, v\in V_{h},\label{eq:ds2-1}\\
 & \inp[\sigma_{h}^{n+1}-\sigma_{h}^{n}]{\xi}=0, & \forall \, \xi \in\W_{h},\label{eq:ds3-1}\\
  & \inp[\K(z_{h}^{n+1}-z_{h}^{n})]{q}-\inp[p_{h}^{n+1}-p_{h}^{n}]{\dvr{q}}=0,
  & \forall \, q \in Z_{h}.
  \label{eq:ds4-1}
\end{align}
Taking $\tau=\sigma_{h}^{n+1}-\sigma_{h}^{n}$, $v=u_{h}^{n+1}-u_{h}^{n}$
and $\xi=\gamma_{h}^{n+1}-\gamma_{h}^{n}$ in \eqref{eq:ds1-1}--\eqref{eq:ds3-1} and summing gives
\begin{equation*}
\inp[A(\sigma_{h}^{n+1}-\sigma_{h}^{n})]{\sigma_{h}^{n+1}-\sigma_{h}^{n}} 
= -\inp[A\a(p_{h}^{n}-p_{h}^{n-1})I]{\sigma_{h}^{n+1}-\sigma_{h}^{n}},
\end{equation*}
implying
\begin{equation}\label{eq:ds-st-2}
\|A^{\frac{1}{2}}(\sigma_{h}^{n+1}-\sigma_{h}^{n})\| \le \alpha \|A^{\frac{1}{2}}(p_{h}^{n}-p_{h}^{n-1})I\|.
\end{equation}
Taking $q=z_{h}^{n+1}$ in  (\ref{eq:ds4-1}) and $w=p_{h}^{n+1}-p_{h}^{n}$
in  (\ref{eq:ds5}) and summing results in
\begin{align*}
& c_{0}\inp[\frac{p_{h}^{n+1}-p_{h}^{n}}{\Dt}]{p_{h}^{n+1}-p_{h}^{n}}
  + \a\inp[A\a\frac{p_{h}^{n+1}-p_{h}^{n}}{\Dt}I]{(p_{h}^{n+1}-p_{h}^{n})I}
  +\inp[\K(z_{h}^{n+1}-z_{h}^{n})]{z_{h}^{n+1}} \\
  & =\a\inp[A\frac{\sigma_{h}^{n+1}-\sigma_{h}^{n}}{\Dt}]{(p_{h}^{n+1}-p_{h}^{n})I}\le\frac{1}{2\Delta t}\|A^{\frac{1}{2}}(\sigma_{h}^{n+1}-\sigma_{h}^{n})\|^{2}
  +\frac{\alpha^{2}}{2\Delta t} \|A^{\frac{1}{2}}(p_{h}^{n+1}-p_{h}^{n})I\|^2,
\end{align*}
which, combined with \eqref{eq:ds-st-2}, implies
\begin{align*}
&  \frac{c_{0}}{\Delta t}\|p_{h}^{n+1}-p_{h}^{n}\|^{2}
  +\frac{\a^{2}}{2\Delta t} \|A^{\frac{1}{2}}(p_{h}^{n+1}-p_{h}^{n})I\|^2
  +\frac{1}{2}\left(\|K^{-\frac12}(z_{h}^{n+1}-z_{h}^{n})\|^{2}
  + \|K^{-\frac12}z_{h}^{n+1}\|^{2} - \|K^{-\frac12}z^{n}_h\|^{2}\right) \\
& \qquad  \le \frac{\alpha^{2}}{2\Delta t} \|A^{\frac{1}{2}}(p_{h}^{n}-p_{h}^{n-1})I\|^2.
\end{align*}
Summing over $n$ from 0 to $k-1$ for any $k = 1,\ldots,N$ and using that $p_h^{-1} = p_h^0$ 
gives
\begin{align*}
  \sum_{n=0}^{k-1}\frac{2c_{0}}{\Delta t}\|p_{h}^{n+1}-p_{h}^{n}\|^{2}
  +\frac{\a^{2}}{\Delta t}\|A^{\frac{1}{2}}(p_{h}^{k}-p_{h}^{k-1})I\|^{2}
  + \|K^{-\frac12}z_{h}^{k}\|^2
  + \sum_{n=0}^{k-1}\|K^{-\frac12}(z_{h}^{n+1}-z_{h}^{n})\|^{2}
  \le\|K^{-\frac12}z_{h}^{0}\|^2. \label{eq:ds-st-3}
\end{align*}
We note that the second and fourth terms
are suboptimal with respect to $\Delta t$. Neglecting these terms and using
\eqref{eq:coercivity-flow}, we obtain
\begin{equation}\label{eq:ds-st-3}
  \sum_{n=0}^{k-1}\frac{c_{0}}{\Delta t}\|p_{h}^{n+1}-p_{h}^{n}\|^{2}
  + \|z_{h}^{k}\|^2 \le C \|z_{h}^{0}\|^2, \quad  k = 1,\ldots,N.
\end{equation}
To obtain control on $p_h$ independent of $c_0$, we use the inf-sup condition
\eqref{eq:inf-sup-darcy} and \eqref{eq:ds4}:
\begin{equation}\label{eq:p-bound}
  \|p_h^{n+1}\| \le C\sup_{0\ne q\in Z_{h}}\frac{(\dvr q,p_h^{n+1})}{\|q\|_{\text{div}}}
  = C \sup_{0\ne q\in Z_{h}}\frac{(\K z_{h}^{n+1},q)}{\|q\|_{\text{div}}} \le C \|z_h^{n+1}\|,
  \quad  n = 0,\ldots,N-1.
\end{equation}
Taking $\tau = \sigma_h^{n+1}$, $v = u_h^{n+1}$, and $\xi = \gamma_h^{n+1}$ in
\eqref{eq:ds1}--\eqref{eq:ds3} gives
\begin{equation}\label{eq:sigma-bound}
\|A^{1/2}\sigma_h^{n+1}\| \le C \|p_h^n\|, \quad  n = 0,\ldots,N-1.
  \end{equation}
For the stability of $u_{h}$ and $\gamma_{h}$, the  inf-sup condition
(\ref{eq:inf-sup-elast}) combined with (\ref{eq:ds1}) gives:
\begin{align}
  \|u_{h}^{n+1}\|+\|\gamma_{h}^{n+1}\|
  & \le C \sup_{0\ne\tau\in\mathbb{X}_{h}}\frac{\inp[u_{h}^{n+1}]{\dvr{\tau}}
    +\inp[\gamma_{h}^{n+1}]{\tau}}{\|\tau\|_{\text{div}}}
  = -C \sup_{0\ne\tau\in\mathbb{X}_{h}}\frac{\inp[A\sigma_{h}^{n+1}]{\tau}
    +\inp[A\a p_{h}^{n}I]{\tau}}{\|\tau\|_{\text{div}}} \nonumber \\
  & \le C \left(\|A^{\frac{1}{2}}\sigma_{h}^{n+1}\| + \|p_{h}^{n}\|\right), \quad n = 0,\ldots,N-1.
  \label{eq:u-gamma-bound}
\end{align}
A combination of bounds \eqref{eq:ds-st-3}--\eqref{eq:u-gamma-bound}
completes the proof of the theorem.

\end{proof}

\subsection{Fixed stress}

The FS decoupling method solves the flow problem first, with
the value of $\sigma$ fixed from the previous time step.  After
that, the mechanics problem is solved using the new values of the
pressure as data \cite{KimFS2011}. We again assume in the analysis
zero source terms for simplicity. The method is: for $n=-1,0,\ldots,N-1$,
find $(\sigma_{h}^{n+1},u_{h}^{n+1},\g_{h}^{n+1},z_{h}^{n+1},p_{h}^{n+1})\in\X_{h}
\times V_{h}\times\W_{h}\times Z_{h}\times W_{h}$ such that
\begin{align}
 & \inp[\K z_{h}^{n+1}]{q}-\inp[p_{h}^{n+1}]{\dvr{q}}=0, & \forall q\in Z_{h},\label{eq:fs4}\\
  & c_{0}\inp[\frac{p_{h}^{n+1}-p_{h}^{n}}{\Dt}]{w}+\a\inp[A\a\frac{p_{h}^{n+1}-p_{h}^{n}}{\Dt}I]{w I}
  +\inp[\dvr{z_{h}^{n+1}}]{w} \nonumber \\
  & \qquad\qquad =-\a\inp[A\frac{\sigma_{h}^{n}-\sigma_{h}^{n-1}}{\Dt}]{wI}, & \forall w\in W_{h},\label{eq:fs5}
\end{align}
and
\begin{align}
 & \inp[A\sigma_{h}^{n+1}]{\tau}+\inp[u_{h}^{n+1}]{\dvr{\tau}}+\inp[\gamma_{h}^{n+1}]{\tau}=-\inp[A\a p_{h}^{n+1}I]{\tau}, & \forall\tau\in\X_{h},\label{eq:fs1}\\
 & \inp[\dvr{\sigma_{h}^{n+1}}]{v}=0, & \forall v\in V_{h},\label{eq:fs2}\\
 & \inp[\sigma_{h}^{n+1}]{\xi}=0, & \forall\xi\in\W_{h},\label{eq:fs3}
\end{align}
where the equations \eqref{eq:fs4} and \eqref{eq:fs1}--\eqref{eq:fs3} hold for
$n=-1,0,\ldots,N-1$ and \eqref{eq:fs5} holds for $n=0,\ldots,N-1$
with $\sigma_{h}^{-1}:=\sigma_{h}^{0}$. Solving \eqref{eq:fs4} and
\eqref{eq:fs1}--\eqref{eq:fs3} for $n=-1$ provides initial data
$\sigma_h^0$, $u_h^0$, $\g_h^0$, and $z_h^0$. 

\subsubsection{Stability analysis for fixed stress}

The following theorem shows that the fixed stress scheme is unconditionally stable.

\begin{thm}\label{thm:fs}
For the solution
  $(\sigma_{h}^{n+1},u_{h}^{n+1},\g_{h}^{n+1},z_{h}^{n+1},p_{h}^{n+1})_{0\le n\le N-1}$
of the system  (\ref{eq:fs4})--(\ref{eq:fs3}), there exists a constant $C$ independent of
$h$, $\Delta t$, $c_0$, and $a_{\min}$ such that
\begin{align*}
&  \sum_{n=0}^{N-1}\frac{c_{0}}{\Delta t}\|p_{h}^{n+1}-p_{h}^{n}\|^{2}
  + \max_{0 \le n \le N-1}\left(\|z_{h}^{n+1}\|^2 + \|p_h^{n+1}\|^2 + \|A^{1/2}\sigma_h^{n+1}\|^2
  + \|u_{h}^{n+1}\|^2 +\|\gamma_{h}^{n+1}\|^2 \right) 
\le C \|z_h^0\|^2.
\end{align*}
\end{thm}
\begin{proof}
The proof is similar to that of the drained split scheme.
Taking the difference of two successive time steps for equations
(\ref{eq:fs1})--(\ref{eq:fs3}) and  (\ref{eq:fs4}), we obtain, for $n=0,\ldots,N-1$,
\begin{align}
  & \inp[A(\sigma_{h}^{n+1}-\sigma_{h}^{n})]{\tau}+\inp[u_{h}^{n+1}-u_{h}^{n}]{\dvr{\tau}}
  +\inp[\gamma_{h}^{n+1}-\gamma_{h}^{n}]{\tau} \nonumber \\
  & \qquad\qquad +\inp[A\a(p_{h}^{n+1}-p_{h}^{n})I]{\tau}=0, & \forall\tau\in\X_{h},\label{eq:fs1-1}\\
 & \inp[\dvr({\sigma_{h}^{n+1}-\sigma_{h}^{n}})]{v}=0, & \forall v\in V_{h},\label{eq:fs2-1}\\
 & \inp[\sigma_{h}^{n+1}-\sigma_{h}^{n}]{\xi}=0, & \forall\xi\in\W_{h},\label{eq:fs3-1}\\
 & \inp[\K(z_{h}^{n+1}-z_{h}^{n})]{q}-\inp[p_{h}^{n+1}-p_{h}^{n}]{\dvr{q}}=0, & \forall q\in Z_{h},\label{eq:fs4-1}
\end{align}
Taking
$\tau=\sigma_{h}^{n+1}-\sigma_{h}^{n}$, $v=u_{h}^{n+1}-u_{h}^{n}$
and $\xi=\gamma_{h}^{n+1}-\gamma_{h}^{n}$ in (\ref{eq:fs1-1})--(\ref{eq:fs3-1}) and adding the
equations results in
\begin{equation}\label{eq:fs-st-2}
\|A^{\frac{1}{2}}(\sigma_{h}^{n+1}-\sigma_{h}^{n})\| \le \alpha \|A^{\frac{1}{2}}(p_{h}^{n+1}-p_{h}^{n})I\|.
\end{equation}
Taking test functions $q=z_{h}^{n+1}$ in  (\ref{eq:fs4-1}) and $w=p_{h}^{n+1}-p_{h}^{n}$
in  (\ref{eq:fs5}) and adding the equations gives
\begin{align*}
& c_{0}\inp[\frac{p_{h}^{n+1}-p_{h}^{n}}{\Dt}]{p_{h}^{n+1}-p_{h}^{n}}
  + \a\inp[A\a\frac{p_{h}^{n+1}-p_{h}^{n}}{\Dt}I]{(p_{h}^{n+1}-p_{h}^{n})I}
  +\inp[\K(z_{h}^{n+1}-z_{h}^{n})]{z_{h}^{n+1}} \\
  & =\a\inp[A\frac{\sigma_{h}^{n}-\sigma_{h}^{n-1}}{\Dt}]{(p_{h}^{n+1}-p_{h}^{n})I}
  \le\frac{1}{2\Delta t}\|A^{\frac{1}{2}}(\sigma_{h}^{n}-\sigma_{h}^{n-1})\|^{2}
  +\frac{\alpha^{2}}{2\Delta t} \|A^{\frac{1}{2}}(p_{h}^{n+1}-p_{h}^{n})I\|^2,
\end{align*}
which, combined with \eqref{eq:fs-st-2}, implies, for $n = 0,\ldots,N-1$,
\begin{align*}
&  \frac{c_{0}}{\Delta t}\|p_{h}^{n+1}-p_{h}^{n}\|^{2}
  +\frac{\a^{2}}{2\Delta t} \|A^{\frac{1}{2}}(p_{h}^{n+1}-p_{h}^{n})I\|^2
  +\frac{1}{2}\left(\|K^{-\frac12}(z_{h}^{n+1}-z_{h}^{n})\|^{2}
  + \|K^{-\frac12}z_{h}^{n+1}\|^{2} - \|K^{-\frac12}z^{n}_h\|^{2}\right) \\
& \qquad  \le \frac{\alpha^{2}}{2\Delta t} \|A^{\frac{1}{2}}(p_{h}^{n}-p_{h}^{n-1})I\|^2,
\end{align*}
where for $n=0$ we have set $p_h^{-1}:=p_h^0$.
Summing over $n$ from 0 to $k-1$ for any $k = 1,\ldots,N$ gives
\begin{equation}\label{eq:fs-st-3}
  \sum_{n=0}^{k-1}\frac{c_{0}}{\Delta t}\|p_{h}^{n+1}-p_{h}^{n}\|^{2}
  + \|z_{h}^{k}\|^2 \le C \|z_{h}^{0}\|^2, \quad  k = 1,\ldots,N.
\end{equation}
Next, similarly to the arguments in Theorem~\ref{thm:ds}, we obtain
\begin{equation}\label{eq:p-bound-fs}
  \|p_h^{n+1}\| \le C \|z_h^{n+1}\|, \quad  n = 0,\ldots,N-1,
\end{equation}
\begin{equation}\label{eq:sigma-bound-fs}
\|A^{1/2}\sigma_h^{n+1}\| \le C \|p_h^{n+1}\|, \quad  n = 0,\ldots,N-1.
  \end{equation}
and
\begin{equation}\label{eq:u-gamma-bound-fs}
  \|u_{h}^{n+1}\|+\|\gamma_{h}^{n+1}\|
  \le C \left(\|A^{\frac{1}{2}}\sigma_{h}^{n+1}\| + \|p_{h}^{n+1}\|\right), \quad n = 0,\ldots,N-1.
\end{equation}
The proof is completed by combining \eqref{eq:fs-st-3}--\eqref{eq:u-gamma-bound-fs}.
\end{proof}

\subsection{Domain decomposition for the split methods}

In this subsection, we present a non-overlapping domain decomposition
method for the drained split decoupled formulation discussed in
subsection \ref{subsec:Drained-Split}, with non-zero source terms.  The
domain decomposition algorithm for the fixed stress decoupled
formulation is similar; it can be obtained by modifying the order of
the coupling terms accordingly. We omit the details.

Following the notation used in Section~\ref{subsec:Monolithic-MFE-Domain}
for the monolithic domain decomposition method, the domain decomposition
method for the DS formulation with non-zero source terms reads as follows: for
$1\le i\le m$ and $n=0,\ldots,N-1$, find $(\sigma_{h,i}^{n+1},u_{h,i}^{n+1},\g_{h,i}^{n+1},\lambda_{h}^{u,n+1})\in\X_{h,i}\times V_{h,i}\times\W_{h,i}\times\Lambda_{h}^{u}$
and $(z_{h,i}^{n+1},p_{h,i}^{n+1},\lambda_{h}^{p,n+1})\in Z_{h,i}\times W_{h,i}\times\Lambda_{h}^p$
such that:

\begin{align*}
  & \inp[A\sigma_{h,i}^{n+1}]{\tau}_{\O_{i}}+\inp[u_{h,i}^{n+1}]{\dvr{\tau}}_{\O_{i}}
  +\inp[\gamma_{h,i}^{n+1}]{\tau}_{\O_{i}} \\
  & \qquad\quad =\inp[A\a p_{h,i}^{n}I]{\tau}_{\O_{i}}
  +\gnp[g_{u}^{n+1}]{\t\,n_i}_{\dO_i \cap\Gd^u}
  +\gnp[\lambda_{h}^{u,n+1}]{\t\,n_i}_{\Gamma_{i}}, &  & \forall\tau\in\X_{h,i},\\
 & \inp[\dvr{\sigma_{h,i}^{n+1}}]{v}_{\O_{i}}=-\inp[f^{n+1}]{v}_{\O_{i}}, &  & \forall v\in V_{h,i},\\
 & \inp[\sigma_{h,i}^{n+1}]{\xi}_{\O_{i}}=0, &  & \forall\xi\in\W_{h,i},\\
 & \sum_{i=1}^{m}\inp[\sigma_{h,i}^{n+1}\,n_{i}]{\mu^{u}}_{\G_{i}}=0, &  & \forall\mu^{u}\in\Lambda_{h}^{u},
\end{align*}
and 
\begin{align*}
 & \inp[\K z_{h,i}^{n+1}]{q}_{\O_{i}}-\inp[p_{h,i}^{n+1}]{\dvr{q}}_{\O_{i}}=-\gnp[g_{p}^{n+1}]{q\cdot n_i}_{\dO_i \cap\Gd^p}-\gnp[\lambda_{h}^{p,n+1}]{q\cdot n_i}_{\Gamma_{i}}, &  & \forall q\in Z_{h,i},\\
  & c_{0}\inp[\frac{p_{h,i}^{n+1}-p_{h,i}^{n}}{\Delta t}]{w}_{\O_{i}}
  +\a\inp[A\a\frac{p_{h,i}^{n+1}-p_{h,i}^{n}}{\Delta t}I]{w I}_{\O_{i}}+\inp[\dvr{z_{h,i}^{n+1}}]{w}_{\O_{i}}\\
 & \hspace{6em} = -\a\inp[\frac{A\left(\sigma_{h,i}^{n+1}-\sigma_{h,i}^{n}\right)}{\Delta t}]{wI}_{\O_{i}}+\inp[g^{n+1}]{w}_{\O_{i}}, &  & \forall w\in W_{h,i},\\
 & \sum_{i=1}^{m}\inp[z_{h,i}^{n+1}\cdot n_{i}]{\mu^{p}}_{\G_{i}}=0, &  & \forall\mu^{p}\in\Lambda_{h}^{p}.
\end{align*}
The above split domain decomposition formulation consists of separate
domain decomposition methods for mechanics and flow at each time
step. Such methods have been studied in detail for the flow
\cite{glowinski1988domain} and mechanics \cite{eldar_elastdd}
components. It is shown that in both cases the global
problem can be reduced to an interface problem with a symmetric and
positive definite operator with condition number $O(h^{-1})$. Therefore,
we employ the conjugate gradient (CG) method for the solution of the
interface problem in each case.

\section{Numerical results}\label{sec:numer}

In this section we report the results of several numerical tests
designed to verify and compare the convergence, stability, and
efficiency of the three domain decomposition methods developed in the
previous sections. The numerical schemes are implemented using deal.II
finite element package \cite{dealII90,BangerthHartmannKanschat2007}.

In all examples the computational domain is the unit square
$(0,1)^{2}$ and the mixed finite element spaces are $\X_{h}\times
V_{h}\times\W_{h} = \mathcal{BDM}_{1}^{2}\times Q_{0}^{2}\times
Q_{0}$ \cite{ArnAwaQiu} for elasticity and $Z_{h}\times
  W_{h}=\mathcal{BDM}_{1}\times Q_{0}$ \cite{brezzi1991mixed} for
  Darcy on quadrilateral meshes. Here $Q_k$ denotes polynomials
  of degree $k$ in each variable. For solving
the interface problem in the monolithic scheme we use non-restarted
unpreconditioned GMRES and in the sequential decoupled methods we use
unpreconditioned CG for the flow and mechanics parts separately. We
use a tolerance on the relative residual $\frac{r_{k}}{r_{o}}$ as
the stopping criteria for both iterative solvers. For Examples 1 and
2, the tolerance is taken to be $10^{-12}$. For Example 3, the
tolerance is taken to be $10^{-6}$ due to relatively smaller initial
residual $r_{0}$. For the monolithic method, Theorem \ref{thm:intr-op-bound}
implies that that the spectral ratio
$\frac{\lambda_{\max}}{\lambda_{\min}}  =
\mathcal{O}(h^{-1})$, where $\lambda_{\min}$ and $\lambda_{\max}$ are the smallest and
largest real eigenvalues of the interface operator, respectively.
Depending on the deviation of the operator
from a normal matrix \cite{kelley1995iterative,IpsenGMRES}, the
growth rate for the number of iterations required for GMRES to converge
could be bounded. In particular, if the interface operator is normal, then
the expected growth rate of the number of GMRES iterations is
$\mathcal{O}\left(\sqrt{\frac{\lambda_{max}}{\lambda_{min}}}\right)$
\cite{kelley1995iterative},
which in our case is $\mathcal{O}(h^{-0.5})$. On the other hand,
the interface operators in the decoupled mechanics and flow systems in
the DS and FS schemes are symmetric and positive definite 
\cite{glowinski1988domain,eldar_elastdd}. A well known result \cite{kelley1995iterative}
is that the number of CG iterations required for convergence
is $\mathcal{O}(\sqrt{\kappa})$, where $\kappa$ is the condition
number for the interface operator. Furthermore, it is shown in
\cite{DarcyDDCOndition,eldar_elastdd} that the condition numbers
$\kappa_{mech}$ and $\kappa_{flow}$ for the interface operators
corresponding to the mechanics and flow parts respectively are
$\mathcal{O}(h^{-1})$ as well and hence the expected growth rate for
the number of CG iterations is also $\mathcal{O}(h^{-0.5})$.

\subsection{Example 1: convergence and stability}

In this example we test the convergence and stability of the three
domain decomposition schemes. We consider
the analytical solution
\[
p=\exp(t)(\sin(\pi x)\cos(\pi y)+10),
\quad u=\exp(t)\begin{pmatrix}x^{3}y^{4}+x^{2}+\sin((1-x)(1-y))\cos(1-y)\\
(1-x)^{4}(1-y)^{3}+(1-y)^{2}+\cos(xy)\sin(x)
\end{pmatrix}.
\]
The physical and numerical parameters are given in
Table~\ref{tab:Physical-Parameters_ex1}.
Using this information, we derive the right hand side and boundary 
and initial conditions for the system  \eqref{biot-1}--\eqref{biot-bc-2}.
\begin{table}[h]
\begin{centering}
\caption{Example 1, physical and numerical parameters.\label{tab:Physical-Parameters_ex1}}
\begin{tabular}{c|c}
\hline 
Parameter & Value\tabularnewline
\hline
Permeability tensor $(K)$ & $I$\tabularnewline
Lame coefficient $(\mu)$ & $100.0$\tabularnewline
Lame coefficient $(\lambda)$ & $100.0$\tabularnewline
Mass storativity $(c_{0})$ & $1.0, 10^{-3}$\tabularnewline
Biot-Willis constant $(\alpha)$ & $1.0$\tabularnewline
Time step $(\Delta t)$ & $10^{-3},10^{-2},10^{-1}$\tabularnewline
Number of time steps & $100$\tabularnewline
\hline 
\end{tabular}
\par\end{centering}
\end{table}
The global mesh is divided into $\ensuremath{2\times2}$ square
subdomains. We run a sequence of refinements from $\ensuremath{h=1/4}$
to $h=1/64$. The initial grids in the bottom left and top right subdomains are
perturbed randomly, resulting in general quadrilateral elements.
The computed solution for the monolithic scheme with $h = 1/64$ and $\Delta t=10^{-3}$
on the final time step is given in Figure
\ref{fig:Example-1,-Computed_monolithic}. 

To study and compare the convergence and stability of the three
methods, we run tests with time steps $\Delta t=10^{-3},10^{-2}\text{
  and }10^{-1}$.  The results with $c_0 = 1$ are presented in Tables
\ref{tab:Example-1,-Convergence_delta_t=00003D10-3}--\ref{tab:Example-1,--Convergence_delta_t=00003D10-1}. We
report the average number of iterations over 100 time steps. The
numerical errors are relative to the corresponding norms of the exact
solution. We use standard Bochner space notation to denote the
space-time norms. Convergence results for the case with $c_{0} =
0.001$ and $\Delta t = 0.01$ are given in
Table~\ref{tab:Example-1,--c_0_small_Convergence_delta_t=00003D10-2}.

The main observation is that all three methods exhibit growth in the
number of interface iterations at the rate of
$\mathcal{O}(h^{-0.5})$. This is consistent with the theoretical
bounds on the spectrum of the interface operator, cf. the discussion
at the beginning of Section~\ref{sec:numer}. This behavior is robust
with respect to both $\Delta t$ and $c_0$. We further note that in both
split schemes, the Darcy interface solver requires fewer number of iterations
than the elasticity solver. We attribute this to the fact that the Darcy formulation
involves a contribution to the diagonal from the time derivative term, resulting
in a smaller condition number of the interface operator.

Another important
conclusion from the tables is that two split schemes are stable
uniformly in $\Delta t$ and $c_0$, in accordance with
Theorem~\ref{thm:ds} and Theorem~\ref{thm:fs}.

In terms of accuracy,
all three methods yield $\mathcal{O}(h)$ convergence for all variables
in their natural norms, which is optimal convergence for the
approximation of the Biot system with the chosen finite element
spaces, cf. \cite{Lee-Biot-five-field,msfmfe-Biot}. In some cases,
especially for larger $\Delta t$, we observe reduction in the
convergence rate for certain variables due to the effect of the time
discretization and/or splitting errors, most notably for the Darcy
velocity in the fixed stress scheme. The accuracy of the three methods
is comparable for smaller $\Delta t$.

In terms of efficiency, in most cases the total number of flow and elasticity CG iterations in the split schemes is comparable to the number of GMRES iterations in the monolithic scheme. However, the 
split schemes have a clear advantage, due to the more efficient CG interface solver compared to GMRES for the monolithic scheme, as well as the less
costly subdomain problems - single-physics solves versus the
coupled Biot solves in the monolithic scheme.

\begin{figure}[h]
\subfloat[Stress x]
{\includegraphics[width=0.33\columnwidth]{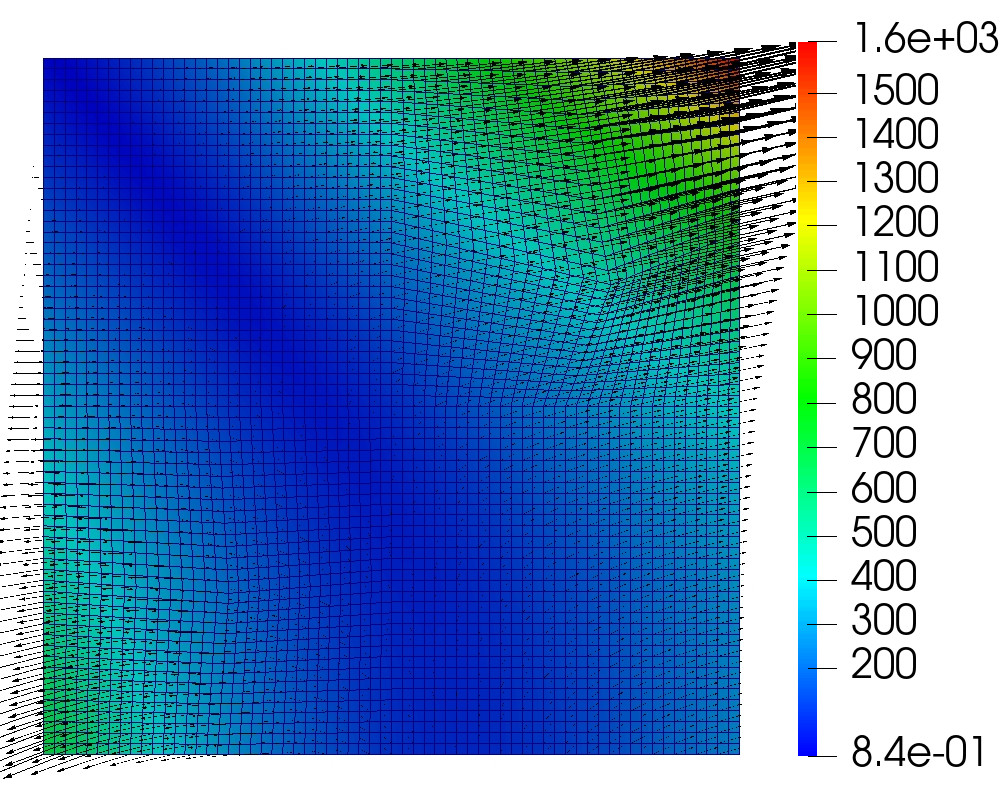}}
~\subfloat[Stress y]
{\includegraphics[width=0.33\columnwidth]{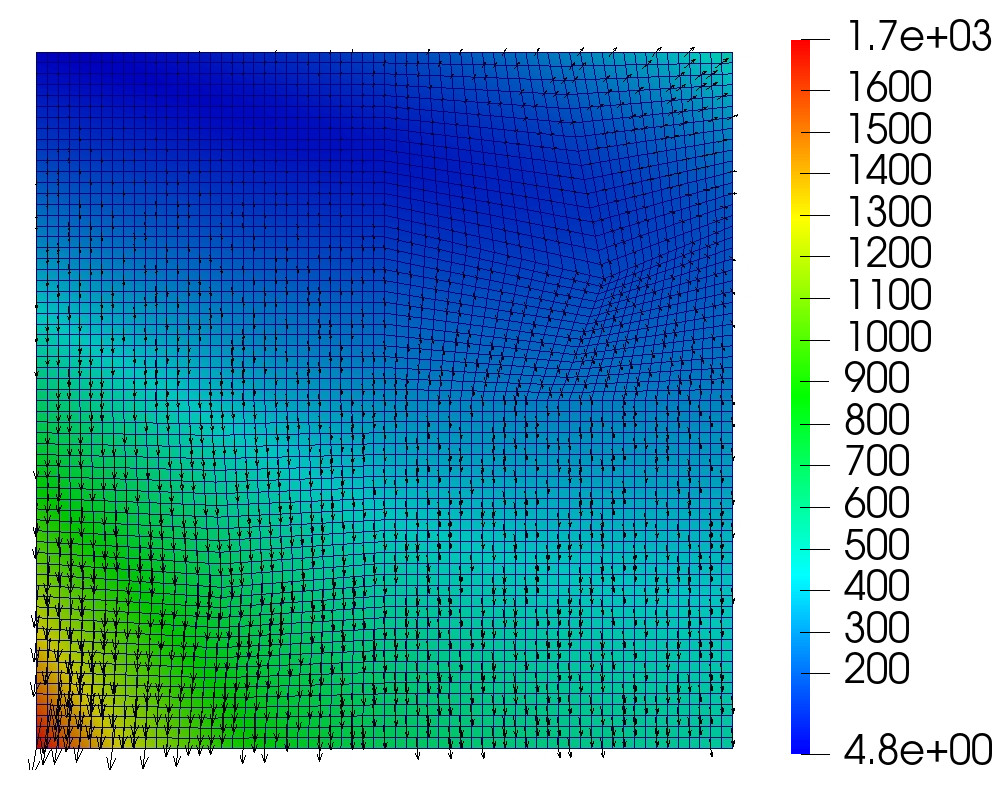}}
~\subfloat[Displacement]
{\includegraphics[width=0.33\columnwidth]{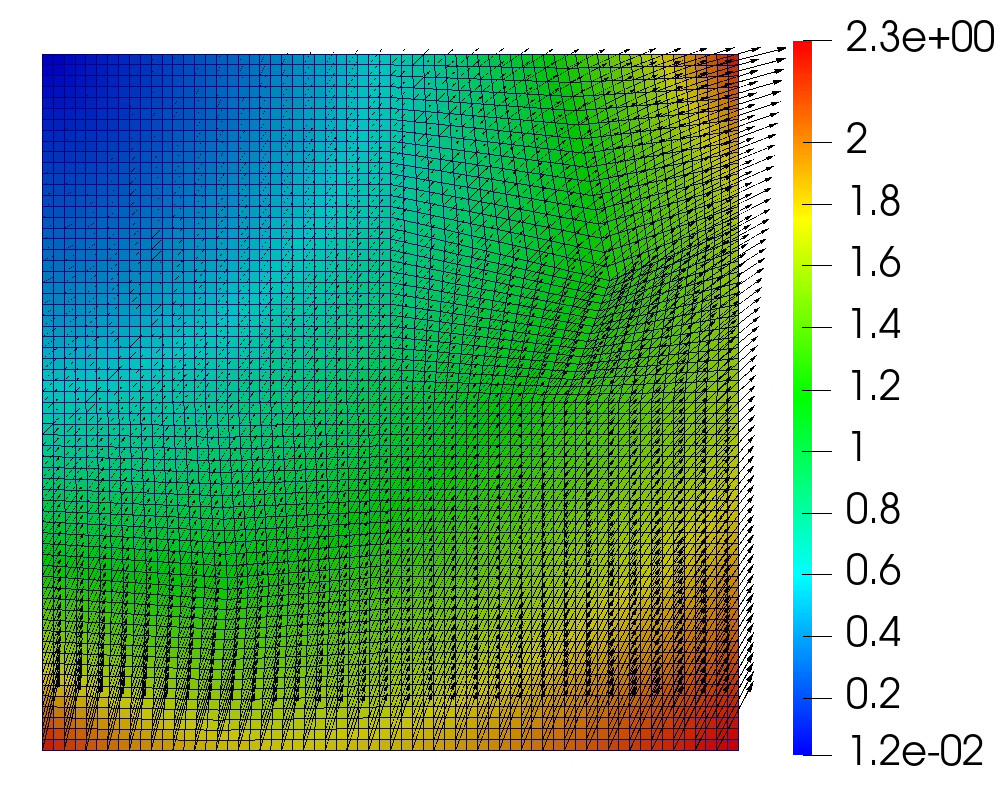}}~

\subfloat[Rotation]{\includegraphics[width=0.33\columnwidth]{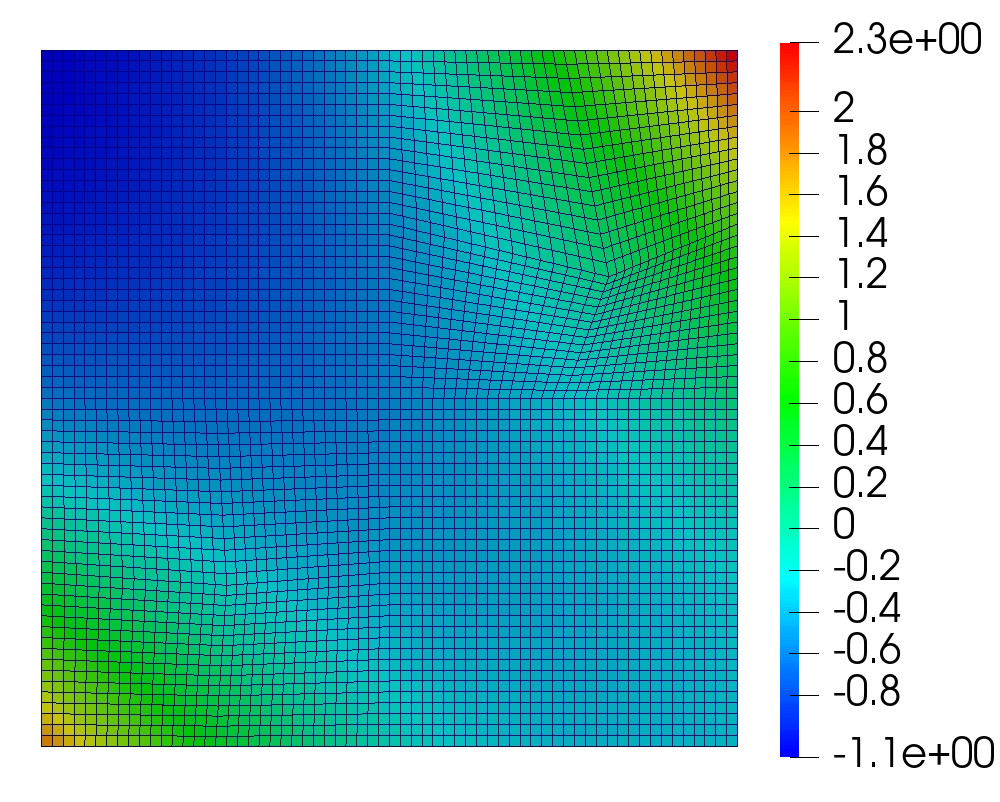}}
\subfloat[Velocity]
{\includegraphics[width=0.33\columnwidth]{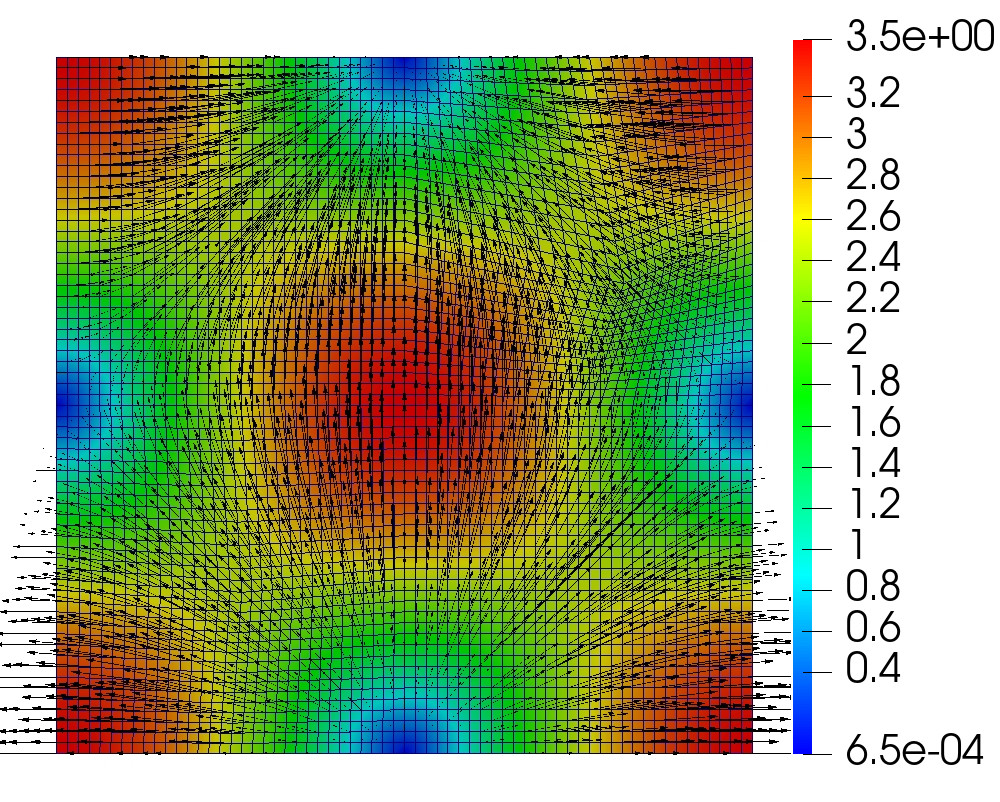}}
~\subfloat[Pressure]{\includegraphics[width=0.33\columnwidth]{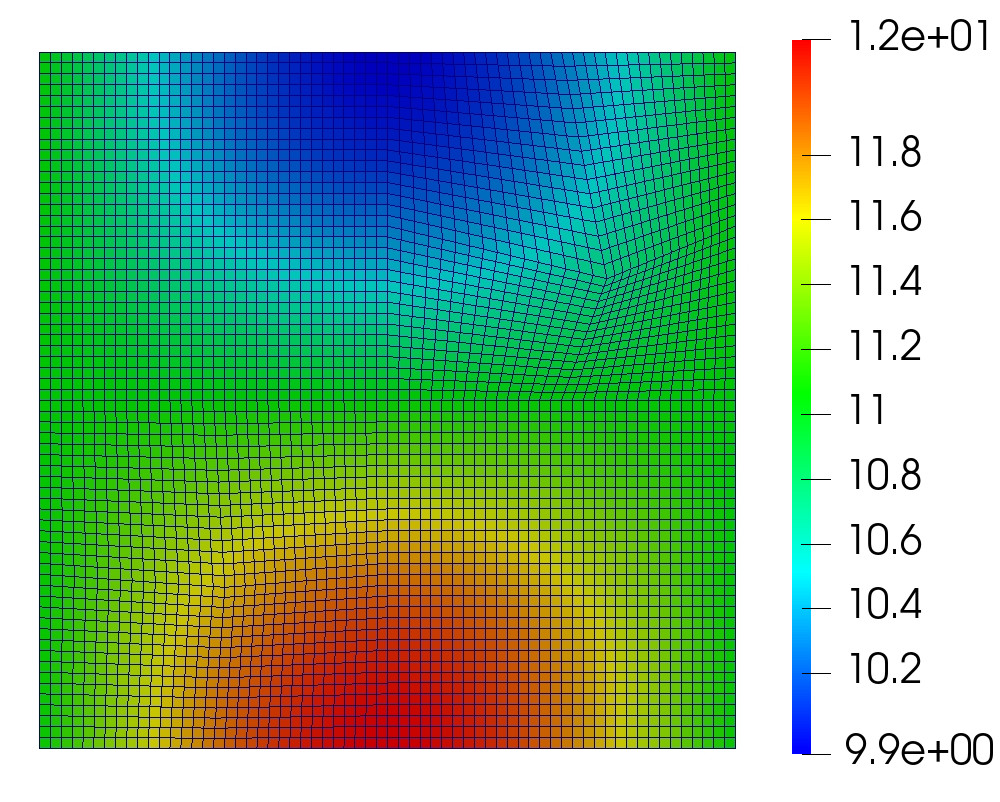}}

\caption{\label{fig:Example-1,-Computed_monolithic}Example
  1, computed solution at the final time step using the monolithic domain
  decomposition method with $h = 1/64$ and $\Delta t=10^{-3}$. }
\end{figure}

\renewcommand{\tabcolsep}{3.9pt}
\begin{table}[h]
  \caption{\label{tab:Example-1,-Convergence_delta_t=00003D10-3}
    Example 1, convergence for $\Delta t=10^{-3}$ and $c_0 = 1$.}
\centering{}%
\subfloat[Monolithic scheme]{
\begin{tabular}{|c|c|c|c|c|c|c|c|c|c|c|}
\hline 
$h$  & \multicolumn{2}{c|}{\#GMRES} & \multicolumn{2}{c|}{$\|z-z_{h}\|_{L^{\infty}(H_{\dvr})}$} & \multicolumn{2}{c|}{$\|p-p_{h}\|_{L^{\infty}(L^{2})}$} & \multicolumn{2}{c|}
    {$\|\sigma-\sigma_{h}\|_{L^{\infty}(H_{\dvr})}$} & \multicolumn{2}{c|}{$\|u-u_{h}\|_{L^{\infty}(L^{2})}$}\tabularnewline
\hline
$1/4$ & 24  & rate & 2.13e+00  & rate & 7.05e-02  & rate & 6.95e-01  & rate & 6.88e-01  & rate\tabularnewline
\cline{3-3} \cline{5-5} \cline{7-7} \cline{9-9} \cline{11-11} 
$1/8$ & 33  & -0.46  & 1.13e+00  & 0.92  & 3.56e-02  & 0.98  & 3.57e-01  & 0.96  & 3.48e-01  & 0.98\tabularnewline
$1/16$ & 44  & -0.42  & 4.84e-01  & 1.22  & 1.79e-02  & 1.00  & 1.79e-01  & 0.99  & 1.75e-01  & 1.00\tabularnewline
$1/32$ & 62  & -0.49  & 2.01e-01  & 1.27  & 8.94e-03  & 1.00  & 8.99e-02  & 1.00  & 8.74e-02  & 1.00\tabularnewline
$1/64$ & 87  & -0.49  & 9.15e-02  & 1.14  & 4.47e-03  & 1.00  & 4.50e-02  & 1.00  & 4.37e-02  & 1.00\tabularnewline
\hline 
\end{tabular}}

\subfloat[Drained split]{
\centering{}%
\begin{tabular}{|c|c|c|c|c|c|c|c|c|c|c|c|c|}
\hline 
$h$  & \multicolumn{2}{c|}{\#CGElast} & \multicolumn{2}{c|}{\#CGDarcy} & \multicolumn{2}{c|}{$\|z-z_{h}\|_{L^{\infty}(H_{\dvr})}$} & \multicolumn{2}{c|}{$\|p-p_{h}\|_{L^{\infty}(L^{2})}$} & \multicolumn{2}{c|}{$\|\sigma-\sigma_{h}\|_{L^{\infty}(H_{\dvr})}$} & \multicolumn{2}{c|}{$\|u-u_{h}\|_{L^{\infty}(L^{2})}$}\tabularnewline
\hline 
$1/4$ & 19  & rate & 10  & rate & 2.00e+00  & rate  & 7.07e-02  & rate & 7.01e-01  & rate & 6.88e-01  & rate\tabularnewline
\cline{3-3} \cline{5-5} \cline{7-7} \cline{9-9} \cline{11-11} \cline{13-13} 
$1/8$ & 23  & -0.28  & 10  & 0.00  & 1.11e+00  & 0.85  & 3.57e-02  & 0.99  & 3.59e-01  & 0.96  & 3.48e-01  & 0.98\tabularnewline
$1/16$ & 34  & -0.56  & 11  & -0.14  & 4.89e-01  & 1.18  & 1.79e-02  & 1.00  & 1.81e-01  & 0.99  & 1.75e-01  & 1.00\tabularnewline
$1/32$ & 47  & -0.47  & 15  & -0.45  & 2.06e-01  & 1.25  & 8.94e-03  & 1.00  & 9.06e-02  & 1.00  & 8.74e-02  & 1.00\tabularnewline
$1/64$ & 65  & -0.47  & 20  & -0.42  & 9.29e-02  & 1.15  & 4.47e-03  & 1.00  & 4.53e-02  & 1.00  & 4.37e-02  & 1.00\tabularnewline
\hline 
\end{tabular}}

\subfloat[Fixed stress]{
\centering{}%
\begin{tabular}{|c|c|c|c|c|c|c|c|c|c|c|c|c|}
\hline 
$h$  & \multicolumn{2}{c|}{\#CGElast} & \multicolumn{2}{c|}{\#CGDarcy} & \multicolumn{2}{c|}{$\|z-z_{h}\|_{L^{\infty}(H_{\dvr})}$} & \multicolumn{2}{c|}{$\|p-p_{h}\|_{L^{\infty}(L^{2})}$} & \multicolumn{2}{c|}{$\|\sigma-\sigma_{h}\|_{L^{\infty}(H_{\dvr})}$} & \multicolumn{2}{c|}{$\|u-u_{h}\|_{L^{\infty}(L^{2})}$}\tabularnewline
\hline 
$1/4$ & 19  & rate & 10  & rate & 1.93e+00  & rate & 7.06e-02  & rate & 7.01e-01  & rate & 6.88e-01  & rate\tabularnewline
\cline{3-3} \cline{5-5} \cline{7-7} \cline{9-9} \cline{11-11} \cline{13-13} 
$1/8$ & 23  & -0.28  & 10  & 0.00  & 1.05e+00  & 0.88  & 3.56e-02  & 0.99  & 3.59e-01  & 0.96  & 3.48e-01  & 0.98\tabularnewline
$1/16$ & 34  & -0.56  & 11  & -0.14  & 4.46e-01  & 1.23  & 1.79e-02  & 1.00  & 1.81e-01  & 0.99  & 1.75e-01  & 1.00\tabularnewline
$1/32$ & 47  & -0.47  & 15  & -0.45  & 2.63e-01  & 0.76  & 8.95e-03  & 1.00  & 9.06e-02  & 1.00  & 8.74e-02  & 1.00\tabularnewline
$1/64$ & 65  & -0.47  & 20  & -0.42  & 2.17e-01  & 0.28  & 4.49e-03  & 0.99  & 4.53e-02  & 1.00  & 4.37e-02  & 1.00\tabularnewline
\hline 
\end{tabular}}
\end{table}

\begin{table}[h]
\caption{\label{tab:Example-1,--Convergence_delta_t=00003D10-2}Example 1,
convergence for $\Delta t=10^{-2}$ and $c_0 = 1$.}
\centering{}%
\subfloat[Monolithic scheme]{
\begin{tabular}{|c|c|c|c|c|c|c|c|c|c|c|}
\hline 
$h$  & \multicolumn{2}{c|}{\#GMRES} & \multicolumn{2}{c|}{$\|z-z_{h}\|_{L^{\infty}(H_{\dvr})}$} & \multicolumn{2}{c|}{$\|p-p_{h}\|_{L^{\infty}(L^{2})}$} & \multicolumn{2}{c|}{$\|\sigma-\sigma_{h}\|_{L^{\infty}(H_{\dvr})}$} & \multicolumn{2}{c|}{$\|u-u_{h}\|_{L^{\infty}(L^{2})}$}\tabularnewline
\hline 
$1/4$ & 18  & rate & 1.58e+00  & rate & 6.98e-02  & rate & 6.97e-01  & rate & 6.88e-01  & rate\tabularnewline
\cline{3-3} \cline{5-5} \cline{7-7} \cline{9-9} \cline{11-11} 
$1/8$ & 23  & -0.35  & 7.47e-01  & 1.08  & 3.55e-02  & 0.97  & 3.58e-01  & 0.96  & 3.48e-01  & 0.98\tabularnewline
$1/16$ & 32  & -0.48  & 3.58e-01  & 1.06  & 1.79e-02  & 0.99  & 1.80e-01  & 0.99  & 1.75e-01  & 0.99\tabularnewline
$1/32$ & 44  & -0.46  & 1.77e-01  & 1.02  & 8.97e-03  & 0.99  & 9.02e-02  & 1.00  & 8.88e-02  & 0.98\tabularnewline
$1/64$ & 63  & -0.52  & 8.98e-02  & 0.98  & 4.54e-03  & 0.98  & 4.53e-02  & 1.00  & 4.66e-02  & 0.93\tabularnewline
\hline 
\end{tabular}}

\subfloat[Drained split]{
\begin{tabular}{|c|c|c|c|c|c|c|c|c|c|c|c|c|}
\hline 
$h$  & \multicolumn{2}{c|}{\#CGElast} & \multicolumn{2}{c|}{\#CGDarcy} & \multicolumn{2}{c|}{$\|z-z_{h}\|_{L^{\infty}(H_{\dvr})}$} & \multicolumn{2}{c|}{$\|p-p_{h}\|_{L^{\infty}(L^{2})}$} & \multicolumn{2}{c|}{$\|\sigma-\sigma_{h}\|_{L^{\infty}(H_{\dvr})}$} & \multicolumn{2}{c|}{$\|u-u_{h}\|_{L^{\infty}(L^{2})}$}\tabularnewline
\hline 
$1/4$ & 19  & rate & 10  & rate  & 1.57e+00  & rate & 6.98e-02  & rate & 7.01e-01  & rate & 6.88e-01  & rate\tabularnewline
\cline{3-3} \cline{5-5} \cline{7-7} \cline{9-9} \cline{11-11} \cline{13-13} 
$1/8$ & 23  & -0.28  & 12  & -0.26  & 7.46e-01  & 1.07  & 3.55e-02  & 0.97  & 3.59e-01  & 0.96  & 3.48e-01  & 0.98\tabularnewline
$1/16$ & 34  & -0.56  & 16  & -0.42  & 3.58e-01  & 1.06  & 1.79e-02  & 0.99  & 1.81e-01  & 0.99  & 1.75e-01  & 1.00\tabularnewline
$1/32$ & 47  & -0.47  & 23  & -0.52  & 1.77e-01  & 1.02  & 8.97e-03  & 0.99  & 9.06e-02  & 1.00  & 8.74e-02  & 1.00\tabularnewline
$1/64$ & 65  & -0.47  & 32  & -0.48  & 8.96e-02  & 0.98  & 4.53e-03  & 0.98  & 4.53e-02  & 1.00  & 4.37e-02  & 1.00\tabularnewline
\hline 
\end{tabular}}

\subfloat[Fixed stress]{
\begin{tabular}{|c|c|c|c|c|c|c|c|c|c|c|c|c|}
\hline 
$h$  & \multicolumn{2}{c|}{\#CGElast} & \multicolumn{2}{c|}{\#CGDarcy} & \multicolumn{2}{c|}{$\|z-z_{h}\|_{L^{\infty}(H_{\dvr})}$} & \multicolumn{2}{c|}{$\|p-p_{h}\|_{L^{\infty}(L^{2})}$} & \multicolumn{2}{c|}{$\|\sigma-\sigma_{h}\|_{L^{\infty}(H_{\dvr})}$} & \multicolumn{2}{c|}{$\|u-u_{h}\|_{L^{\infty}(L^{2})}$}\tabularnewline
\hline 
$1/4$ & 19  & rate & 10  & rate & 1.48e+00  & rate & 6.97e-02  & rate & 7.01e-01  & rate & 6.88e-01  & rate\tabularnewline
\cline{3-3} \cline{5-5} \cline{7-7} \cline{9-9} \cline{11-11} \cline{13-13} 
$1/8$ & 23  & -0.28  & 12  & -0.26  & 7.64e-01  & 0.96  & 3.56e-02  & 0.97  & 3.59e-01  & 0.96  & 3.48e-01  & 0.98\tabularnewline
$1/16$ & 34  & -0.56  & 16  & -0.42  & 4.88e-01  & 0.65  & 1.81e-02  & 0.98  & 1.81e-01  & 0.99  & 1.75e-01  & 1.00\tabularnewline
$1/32$ & 47  & -0.47  & 23  & -0.52  & 3.80e-01  & 0.36  & 9.37e-03  & 0.95  & 9.06e-02  & 1.00  & 8.74e-02  & 1.00\tabularnewline
$1/64$ & 65  & -0.47  & 32  & -0.48  & 3.44e-01  & 0.14  & 5.26e-03  & 0.83  & 4.53e-02  & 1.00  & 4.37e-02  & 1.00\tabularnewline
\hline 
\end{tabular}}
\end{table}

\begin{table}[h]
\caption{\label{tab:Example-1,--Convergence_delta_t=00003D10-1}Example 1,
convergence for $\Delta t=10^{-1}$ and $c_0 = 1$.}
\centering{}%
\subfloat[Monolithic scheme]{
\begin{tabular}{|c|c|c|c|c|c|c|c|c|c|c|}
\hline 
$h$  & \multicolumn{2}{c|}{\#GMRES} & \multicolumn{2}{c|}{$\|z-z_{h}\|_{L^{\infty}(H_{\dvr})}$} & \multicolumn{2}{c|}{$\|p-p_{h}\|_{L^{\infty}(L^{2})}$} & \multicolumn{2}{c|}{$\|\sigma-\sigma_{h}\|_{L^{\infty}(H_{\dvr})}$} & \multicolumn{2}{c|}{$\|u-u_{h}\|_{L^{\infty}(L^{2})}$}\tabularnewline
\hline 
$1/4$ & 40  & rate & 1.38e+00  & rate & 6.99e-02  & rate & 7.04e-01  & rate & 7.17e-01  & rate\tabularnewline
\cline{3-3} \cline{5-5} \cline{7-7} \cline{9-9} \cline{11-11} 
$1/8$ & 59  & -0.56  & 7.20e-01  & 0.94  & 3.63e-02  & 0.94  & 3.65e-01  & 0.95  & 4.26e-01  & 0.75\tabularnewline
$1/16$ & 88  & -0.58  & 3.97e-01  & 0.86  & 1.94e-02  & 0.90  & 1.92e-01  & 0.93  & 3.09e-01  & 0.46\tabularnewline
$1/32$ & 128  & -0.54  & 2.57e-01  & 0.63  & 1.17e-02  & 0.72  & 1.09e-01  & 0.81  & 2.72e-01  & 0.19\tabularnewline
$1/64$ & 180  & -0.49  & 2.08e-01  & 0.31  & 8.84e-03  & 0.41  & 7.40e-02  & 0.56  & 2.62e-01  & 0.06\tabularnewline
\hline 
\end{tabular}}

\subfloat[Drained split]{
\centering{}%
\begin{tabular}{|c|c|c|c|c|c|c|c|c|c|c|c|c|}
\hline 
$h$  & \multicolumn{2}{c|}{\#CGElast} & \multicolumn{2}{c|}{\#CGDarcy} & \multicolumn{2}{c|}{$\|z-z_{h}\|_{L^{\infty}(H_{\dvr})}$} & \multicolumn{2}{c|}{$\|p-p_{h}\|_{L^{\infty}(L^{2})}$} & \multicolumn{2}{c|}{$\|\sigma-\sigma_{h}\|_{L^{\infty}(H_{\dvr})}$} & \multicolumn{2}{c|}{$\|u-u_{h}\|_{L^{\infty}(L^{2})}$}\tabularnewline
\hline 
$1/4$ & 19  & rate  & 11  & rate & 1.38e+00  & rate & 6.99e-02  & rate & 7.01e-01  & rate & 6.88e-01  & rate\tabularnewline
\cline{3-3} \cline{5-5} \cline{7-7} \cline{9-9} \cline{11-11} \cline{13-13} 
$1/8$ & 23  & -0.28  & 14  & -0.35  & 7.17e-01  & 0.94  & 3.62e-02  & 0.95  & 3.59e-01  & 0.96  & 3.48e-01  & 0.98\tabularnewline
$1/16$ & 34  & -0.56  & 20  & -0.51  & 3.92e-01  & 0.87  & 1.92e-02  & 0.91  & 1.81e-01  & 0.99  & 1.75e-01  & 1.00\tabularnewline
$1/32$ & 47  & -0.47  & 28  & -0.49  & 2.50e-01  & 0.65  & 1.15e-02  & 0.75  & 9.07e-02  & 1.00  & 8.74e-02  & 1.00\tabularnewline
$1/64$ & 65  & -0.47  & 38  & -0.44  & 1.99e-01  & 0.33  & 8.48e-03  & 0.44  & 4.56e-02  & 0.99  & 4.37e-02  & 1.00\tabularnewline
\hline 
\end{tabular}}

\subfloat[Fixed stress]{
\centering{}%
\begin{tabular}{|c|c|c|c|c|c|c|c|c|c|c|c|c|}
\hline 
$h$  & \multicolumn{2}{c|}{\#CGElast} & \multicolumn{2}{c|}{\#CGDarcy} & \multicolumn{2}{c|}{$\|z-z_{h}\|_{L^{\infty}(H_{\dvr})}$} & \multicolumn{2}{c|}{$\|p-p_{h}\|_{L^{\infty}(L^{2})}$} & \multicolumn{2}{c|}{$\|\sigma-\sigma_{h}\|_{L^{\infty}(H_{\dvr})}$} & \multicolumn{2}{c|}{$\|u-u_{h}\|_{L^{\infty}(L^{2})}$}\tabularnewline
\hline 
$1/4$ & 19  & rate & 11  & rate & 1.42e+00  & rate & 7.00e-02  & rate & 7.00e-01  & rate & 6.88e-01  & rate\tabularnewline
\cline{3-3} \cline{5-5} \cline{7-7} \cline{9-9} \cline{11-11} \cline{13-13} 
$1/8$ & 23  & -0.28  & 14  & -0.35  & 8.38e-01  & 0.76  & 3.63e-02  & 0.95  & 3.59e-01  & 0.96  & 3.48e-01  & 0.98\tabularnewline
$1/16$ & 34  & -0.56  & 20  & -0.51  & 5.83e-01  & 0.52  & 1.93e-02  & 0.91  & 1.81e-01  & 0.99  & 1.75e-01  & 1.00\tabularnewline
$1/32$ & 47  & -0.47  & 28  & -0.49  & 4.87e-01  & 0.26  & 1.15e-02  & 0.74  & 9.06e-02  & 1.00  & 8.74e-02  & 1.00\tabularnewline
$1/64$ & 65  & -0.47  & 38  & -0.44  & 4.56e-01  & 0.09  & 8.53e-03  & 0.44  & 4.53e-02  & 1.00  & 4.37e-02  & 1.00\tabularnewline
\hline 
\end{tabular}}
\end{table}

\begin{table}[h]
  \caption{\label{tab:Example-1,--c_0_small_Convergence_delta_t=00003D10-2}
    Example 1, convergence for $\Delta t=10^{-2}$ and $c_{0}=10^{-3}$.}
	\centering{}%
	\subfloat[Monolithic scheme]{
	\begin{tabular}{|c|c|c|c|c|c|c|c|c|c|c|}
		\hline 
		$h$  & \multicolumn{2}{c|}{\#GMRES} & \multicolumn{2}{c|}{$\|z-z_{h}\|_{L^{\infty}(H_{\dvr})}$} & \multicolumn{2}{c|}{$\|p-p_{h}\|_{L^{\infty}(L^{2})}$} & \multicolumn{2}{c|}{$\|\sigma-\sigma_{h}\|_{L^{\infty}(H_{\dvr})}$} & \multicolumn{2}{c|}{$\|u-u_{h}\|_{L^{\infty}(L^{2})}$}\tabularnewline
		\hline 
		$h/4$ & 21  & rate & 1.86e+00  & rate & 7.12e-02  & rate & 6.97e-01  & rate & 6.88e-01  & rate\tabularnewline
		\cline{3-3} \cline{5-5} \cline{7-7} \cline{9-9} \cline{11-11} 
		$h/8$ & 28  & -0.42  & 7.87e-01  & 1.24  & 3.57e-02  & 1.00  & 3.58e-01  & 0.96  & 3.48e-01  & 0.98\tabularnewline
		$h/16$ & 38  & -0.44  & 3.63e-01  & 1.12  & 1.79e-02  & 1.00  & 1.80e-01  & 0.99  & 1.75e-01  & 0.99\tabularnewline
		$h/32$ & 53  & -0.48  & 1.77e-01  & 1.04  & 8.94e-03  & 1.00  & 9.02e-02  & 1.00  & 8.88e-02  & 0.98\tabularnewline
		$h/64$ & 73  & -0.46  & 8.78e-02  & 1.01  & 4.47e-03  & 1.00  & 4.53e-02  & 1.00  & 4.66e-02  & 0.93\tabularnewline
		\hline 
	\end{tabular}}
	
	\subfloat[Drained split]{
		\begin{tabular}{|c|c|c|c|c|c|c|c|c|c|c|c|c|}
			\hline 
			$h$  & \multicolumn{2}{c|}{\#CGElast} &
                        \multicolumn{2}{c|}{\#CGDarcy} & \multicolumn{2}{c|}{$\|z-z_{h}\|_{L^{\infty}(H_{\dvr})}$} & \multicolumn{2}{c|}{$\|p-p_{h}\|_{L^{\infty}(L^{2})}$} & \multicolumn{2}{c|}{$\|\sigma-\sigma_{h}\|_{L^{\infty}(H_{\dvr})}$} & \multicolumn{2}{c|}{$\|u-u_{h}\|_{L^{\infty}(L^{2})}$}\tabularnewline
			\hline 
			$1/4$ & 19 & rate & 11 & rate & 1.90e+00 & rate & 7.16e-02 & rate & 7.01e-01 & rate & 6.88e-01 & rate\tabularnewline
			\cline{3-3} \cline{5-5} \cline{7-7} \cline{9-9} \cline{11-11} \cline{13-13} 
			$1/8$ & 23 & -0.28 & 15 & -0.45 & 7.91e-01 & 1.26 & 3.58e-02 & 1.00 & 3.59e-01 & 0.96 & 3.48e-01 & 0.98\tabularnewline
			$1/16$ & 34 & -0.56 & 20 & -0.42 & 3.64e-01 & 1.12 & 1.79e-02 & 1.00 & 1.81e-01 & 0.99 & 1.75e-01 & 1.00\tabularnewline
			$1/32$ & 47 & -0.47 & 28 & -0.49 & 1.78e-01 & 1.03 & 8.98e-03 & 1.00 & 9.06e-02 & 1.00 & 8.74e-02 & 1.00\tabularnewline
			$1/64$ & 65 & -0.47 & 41 & -0.55 & 9.01e-02 & 0.98 & 4.55e-03 & 0.98 & 4.53e-02 & 1.00 & 4.37e-02 & 1.00\tabularnewline
			\hline 
	\end{tabular}}
	
	\subfloat[Fixed stress]{
		\begin{tabular}{|c|c|c|c|c|c|c|c|c|c|c|c|c|}
			\hline 
			$h$  & \multicolumn{2}{c|}{\#CGElast} &
                        \multicolumn{2}{c|}{\#CGDarcy} & \multicolumn{2}{c|}{$\|z-z_{h}\|_{L^{\infty}(H_{\dvr})}$} & \multicolumn{2}{c|}{$\|p-p_{h}\|_{L^{\infty}(L^{2})}$} & \multicolumn{2}{c|}{$\|\sigma-\sigma_{h}\|_{L^{\infty}(H_{\dvr})}$} & \multicolumn{2}{c|}{$\|u-u_{h}\|_{L^{\infty}(L^{2})}$}\tabularnewline
			\hline 
			$1/4$ & 19 & rate & 11 & rate & 1.88e+00 & rate & 7.16e-02 & rate & 7.01e-01 & rate & 6.88e-01 & rate\tabularnewline
			\cline{3-3} \cline{5-5} \cline{7-7} \cline{9-9} \cline{11-11} \cline{13-13}
			$1/8$ & 23 & -0.28 & 15 & -0.45 & 8.84e-01 & 1.09 & 3.72e-02 & 0.94 & 3.59e-01 & 0.96 & 3.48e-01 & 0.98\tabularnewline
			$1/16$ & 34 & -0.56 & 20 & -0.42 & 6.43e-01 & 0.46 & 2.08e-02 & 0.84 & 1.81e-01 & 0.99 & 1.75e-01 & 1.00\tabularnewline
			$1/32$ & 47 & -0.47 & 28 & -0.49 & 5.60e-01 & 0.20 & 1.36e-02 & 0.61 & 9.06e-02 & 1.00 & 8.74e-02 & 1.00\tabularnewline
			$1/64$ & 65 & -0.47 & 41 & -0.55 & 5.36e-01 & 0.06 & 1.10e-02 & 0.31 & 4.53e-02 & 1.00 & 4.37e-02 & 1.00\tabularnewline
			\hline 
	\end{tabular}}
\end{table}


\subsection{Example 2: dependence on number of subdomains}

The objective of this example is to study how the number of GMRES and
CG iterations required for the different schemes depend on the number (and diameter)
of subdomains used in the domain decomposition. For this example, we
use the same test case as in Example 1. We solve the system using 4
$(2\times2)$, 16 $(4\times4)$, and 64 $(8\times8)$ square subdomains
of identical size. The physical parameters are as in Example 1, with
$c_0 = 1$, $\Delta t=10^{-3}$, and $T=100\times\Delta t$. The average
number and growth rate of iterations in the three methods are reported
in Tables
\ref{table-2-1}--\ref{table-2-3},
where $A$ denotes the subdomain diameter. We note that the number of iterations
for the drained split and fixed stress schemes are identical, so we give
one table for both methods. For a fixed $A$, the growth
rate with respect to $h$ is averaged over all mesh refinements.  For a
fixed mesh size $h$, the growth rate with respect to $A$ is averaged
over the different domain decompositions. For all three methods, we
observe that for a fixed number of subdomains, the growth rate in the
number of iterations with respect to mesh refinement is approximately
$\mathscr{\mathcal{O}}(h^{-0.5})$, being slightly better for the Darcy
solver in the split schemes. As this is the same as the growth rate in
Example 1, the conclusion from Example 1 that the growth rate is
consistent with the theory extends to domain decompositions with
varying number of subdomains, see also the discussion at the beginning
of Section~\ref{sec:numer}. We further observe that for a fixed mesh
size, the growth rate in number of iterations with respect to
subdomain diameter $A$ is approximately
$\mathscr{\mathcal{O}}(A^{-0.5})$, again being somewhat better for the
Darcy solves.  This is consistent with theoretical results bounding
the spectral ratio of the unpreconditioned interface operator as
$\mathscr{\mathcal{O}}\left(\left(hA\right)^{-1}\right)$
\cite{Toselli-Widlund}. The dependence on $A$ can be eliminated with
the use of a coarse solve preconditioner
\cite{DarcyDDCOndition,Toselli-Widlund}.

\begin{table}[h]
  \centering{}
  \caption{Example 2, number of GMRES iterations in the monolithic scheme\label{table-2-1}}
\begin{tabular}{c|c|c|c|c}
$h$ & $2\times2$ & $4\times4$ & $8\times8$ & Rate\tabularnewline
\hline 
$1/8$ & 33 & 53 & 76 & $\mathcal{O}(A^{-0.60})$\tabularnewline
$1/16$ & 45 & 68 & 97 & $\mathcal{O}(A^{-0.55})$\tabularnewline
$1/32$ & 63 & 93 & 126 & $\mathcal{O}(A^{-0.50})$\tabularnewline
$1/64$ & 88 & 125 & 164 & $\mathcal{O}(A^{-0.45})$\tabularnewline
\hline 
Rate & $\mathscr{\mathcal{O}}(h^{-0.47})$ & $\mathscr{\mathcal{O}}(h^{-0.41})$ & $\mathscr{\mathcal{O}}(h^{-0.36})$ & \tabularnewline
\end{tabular}
\end{table}

\begin{table}[h]\label{table-2-2}
\begin{singlespace}
\begin{centering}
\caption{Example 2, number of CG elasticity iterations in the drained split and fixed stress schemes}
\begin{tabular}{c|c|c|c|c}
$h$ & $2\times2$ & $4\times4$ & $8\times8$ & Rate\tabularnewline
\hline 
$1/8$ & 23 & 40 & 60 & $\mathcal{O}(A^{-0.69})$\tabularnewline
$1/16$ & 34 & 51 & 73 & $\mathcal{O}(A^{-0.55})$\tabularnewline
$1/32$ & 47 & 68 & 95 & $\mathcal{O}(A^{-0.51})$\tabularnewline
$1/64$ & 65 & 95 & 124 & $\mathcal{O}(A^{-0.46})$\tabularnewline
\hline 
Rate & $\mathcal{O}(h^{-0.50})$ & $\mathcal{O}(h^{-0.42})$ & $\mathcal{O}(h^{-0.35})$ & \tabularnewline
\end{tabular}
\par\end{centering}
\end{singlespace}

\end{table}

\begin{table}[h]
\begin{centering}
  \caption{Example 2, number of CG Darcy iterations in the
    drained split and fixed stress schemes\label{table-2-3}}
\begin{tabular}{c|c|c|c|c}
$h$ & $2\times2$ & $4\times4$ & $8\times8$ & Rate\tabularnewline
\hline 
$1/8$ & 10 & 11 & 14 & $\mathcal{O}(A^{-0.24})$\tabularnewline
$1/16$ & 11 & 12 & 14 & $\mathcal{O}(A^{-0.17})$\tabularnewline
$1/32$ & 15 & 16 & 18 & $\mathcal{O}(A^{-0.13})$\tabularnewline
$1/64$ & 20 & 23 & 24 & $\mathcal{O}(A^{-0.13})$\tabularnewline
\hline 
Rate & $\mathcal{O}(h^{-0.34})$ & $\mathcal{O}(h^{-0.36})$ & $\mathcal{O}(h^{-0.25})$ & \tabularnewline
\end{tabular}
\par\end{centering}
%
\end{table}

\subsection{Example 3: heterogeneous benchmark}

This example illustrates the performance of the methods for highly
heterogeneous media.  We use porosity and permeability fields from the
Society of Petroleum Engineers 10th Comparative Solution Project
(SPE10)\footnote{https://www.spe.org/web/csp/datasets/set02.htm}.  The
computational domain is $\Omega=(0,1)^{2}$, which is partitioned into
a $128\times128$ square grid. We decompose the domain into $4\times4$
square subdomains.  From the porosity field data, the Young's modulus
is obtained using the relation
$E=10^{2}\left(1-\frac{\phi}{c}\right)^{2.1},$ where $c=0.5$, refers
to the porosity at which the Young's modulus vanishes, see
\cite{kovavcik1999correlation} for details. The porosity, Young's
modulus and permeability fields are given in Figure
\ref{fig:Input-fields-for-example3}. The parameters and boundary
conditions are given in Table \ref{tab:Physical-Parameters_ex3}.  The
source terms are taken to be zero. These conditions describe flow from
left to right, driven by a pressure gradient. Since in this example analytical
solution is not available, we need to prescribe suitable initial data.
The initial condition
for the pressure is taken to be $p_0 = 1 - x$, which is compatible with
the prescribed boundary conditions. We then follow the procedure described in
Remark~\ref{rem:init} to obtain discrete initial data. In particular,
we set $p_h^0$ to be the $L^2$-projection of $p_0$ onto $W_h$ and solve
a mixed elasticity domain decomposition problem at $t = 0$ to obtain $\sigma_h^0$. We
note that this solve also gives $u_h^0$, $\gamma_h^0$, and $\lambda_h^{u,0}$.
In the case of the monolithic scheme where the time-differentiated elasticity
equation \eqref{eq:dd1-mfe1-dsc} is solved, the computed initial data is used to recover
 $u_h^n$, $\gamma_h^n$, and $\lambda_h^{u,n}$ using \eqref{init-recover}.
The computed solution
using the monolithic domain decomposition scheme is given in Figure
\ref{fig:Monolithic-Scheme:ex3}.  The solutions from the two split
methods look similar.

\begin{figure}[h]
  \includegraphics[width=0.33\columnwidth]{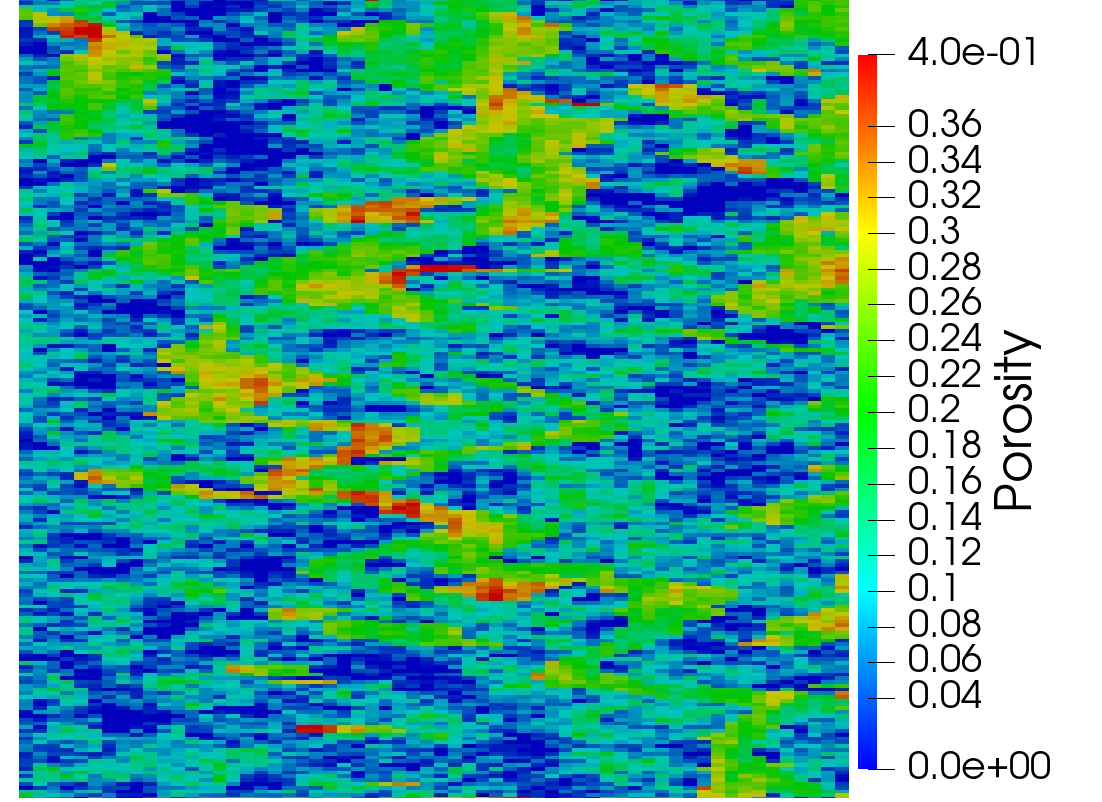}
  \includegraphics[width=0.33\columnwidth]{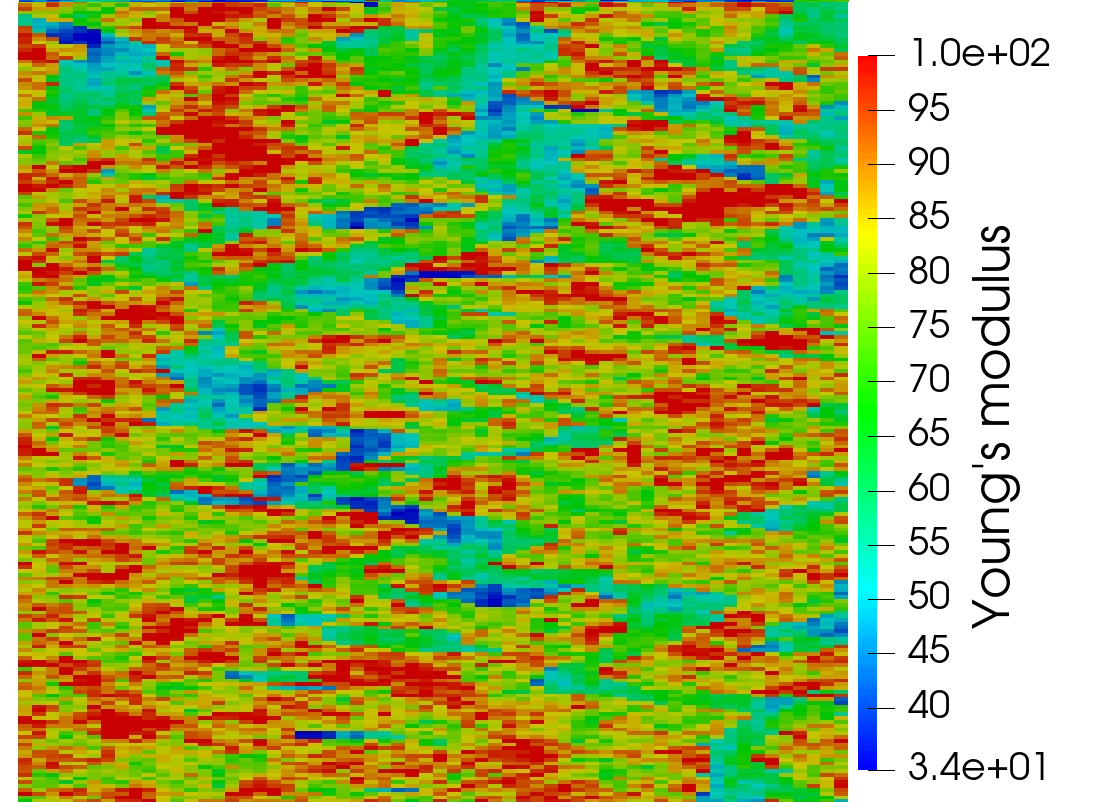}
  \includegraphics[width=0.33\columnwidth]{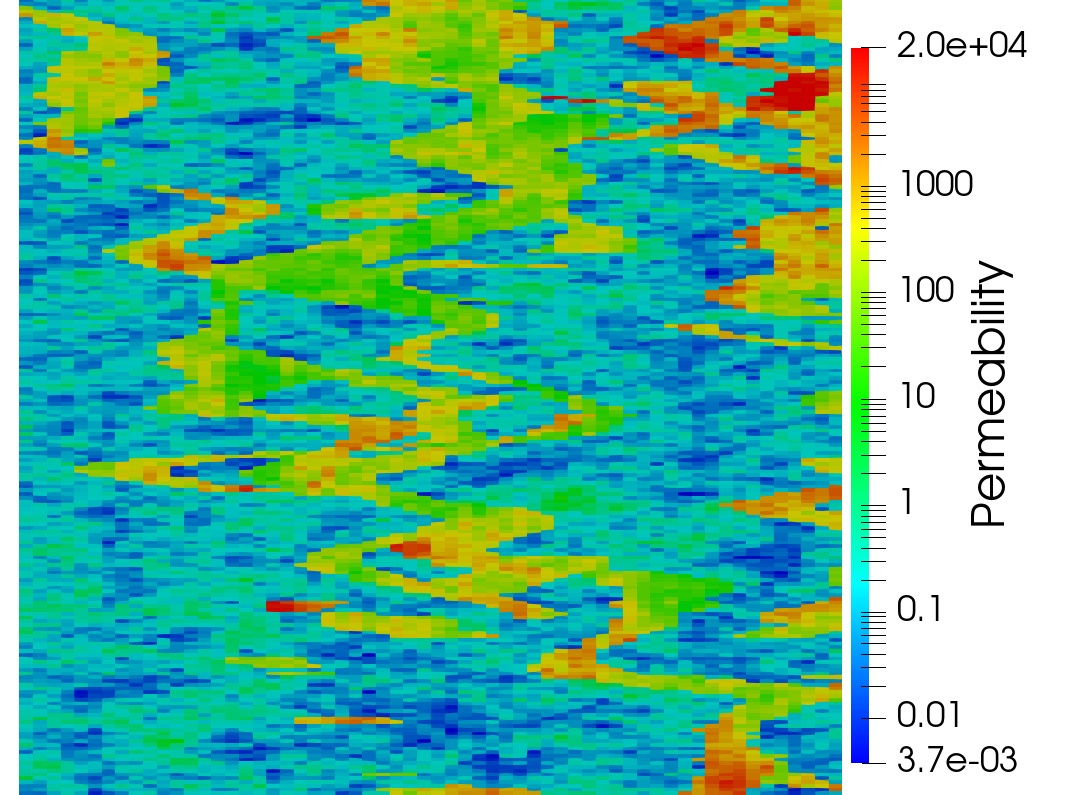}

\caption{Example 3, porosity, Young's modulus, permeability.
  \label{fig:Input-fields-for-example3}}
\end{figure}

\begin{table}[h]
  \centering{}
  \caption{Example 3, parameters (left) and boundary conditions (right)
    \label{tab:Physical-Parameters_ex3}.}
\begin{tabular}{c|c}
Parameter & Value\tabularnewline
\hline 
Mass storativity $(c_{0})$ & $1.0$\tabularnewline
Biot-Willis constant $(\alpha)$ & $1.0$\tabularnewline
Time step $(\Delta t)$ & $10^{-2}$\tabularnewline
Total time $(T)$ & $1.0$\tabularnewline
\end{tabular}
\hspace{5 mm}
\begin{tabular}{c|c|c|c|c}
Boundary & $\sigma$ & $u$ & $p$ & $z$\tabularnewline
\hline 
Left & $\sigma n=-\alpha p n$ & - & $1$ & -\tabularnewline
Bottom & $\sigma n=0$ & - & - & $z\cdot n=0$\tabularnewline
Right & $-$ & $0$ & $0$ & -\tabularnewline
Top & $\sigma n=0$ & - & - & $z\cdot n=0$\tabularnewline
\end{tabular}

\end{table}

\begin{figure}[h]
\begin{centering}
  \subfloat[Pressure]{\includegraphics[width=0.33\columnwidth]{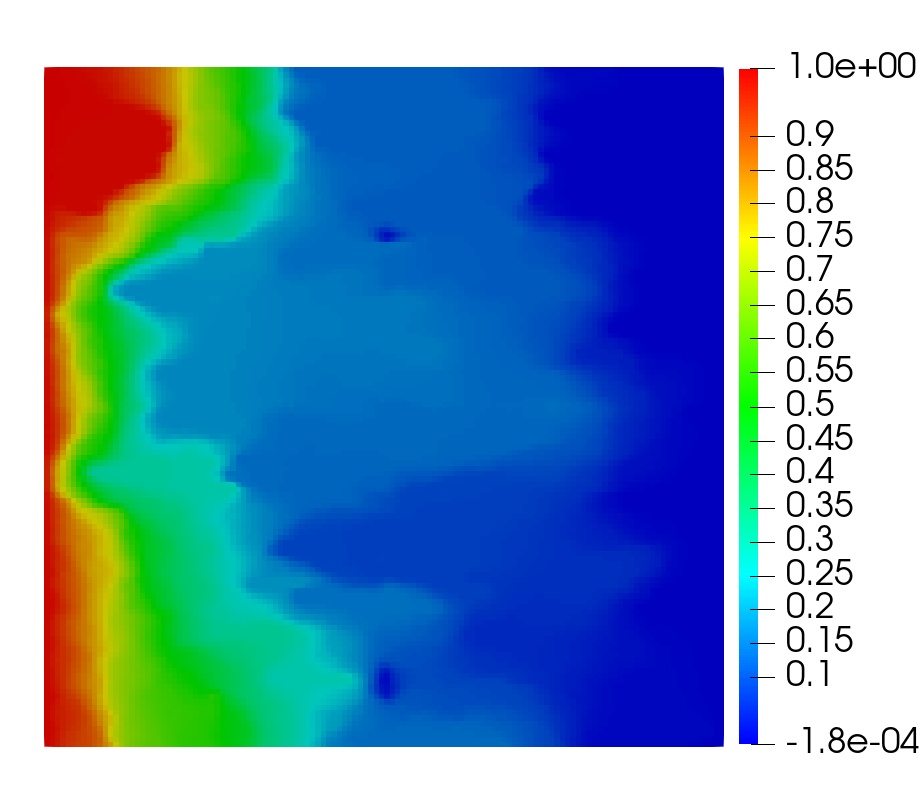}}
  \subfloat[Velocity]{\includegraphics[width=0.33\columnwidth]{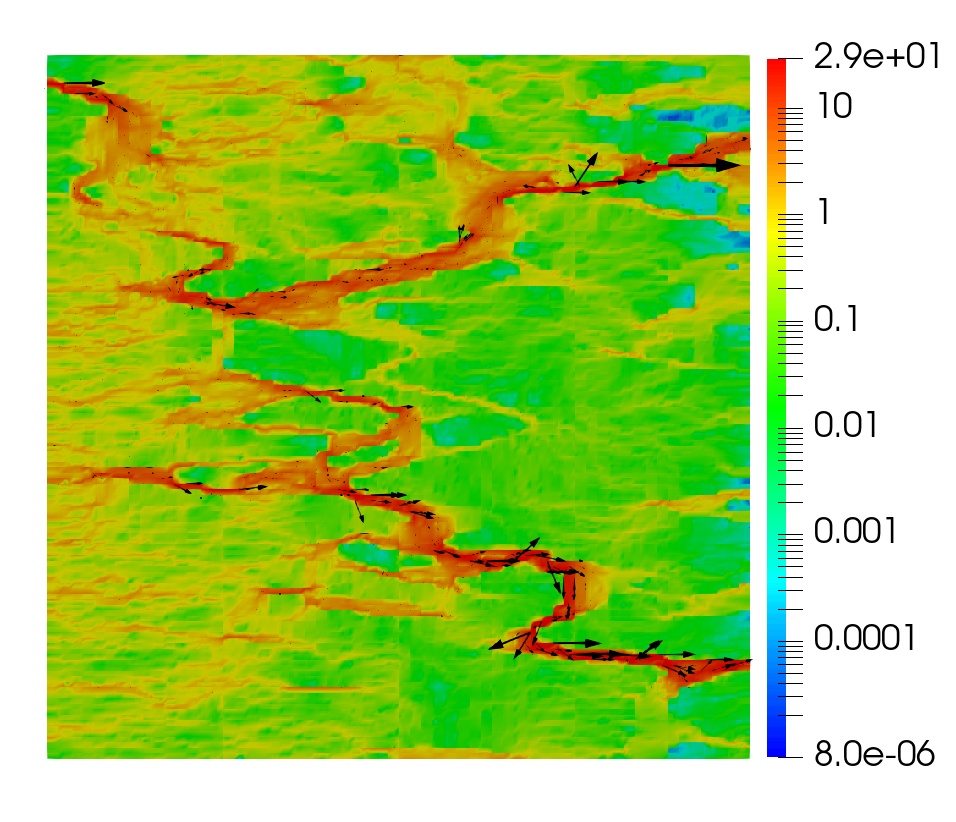}}
  \subfloat[Displacement]{\includegraphics[width=0.33\columnwidth]{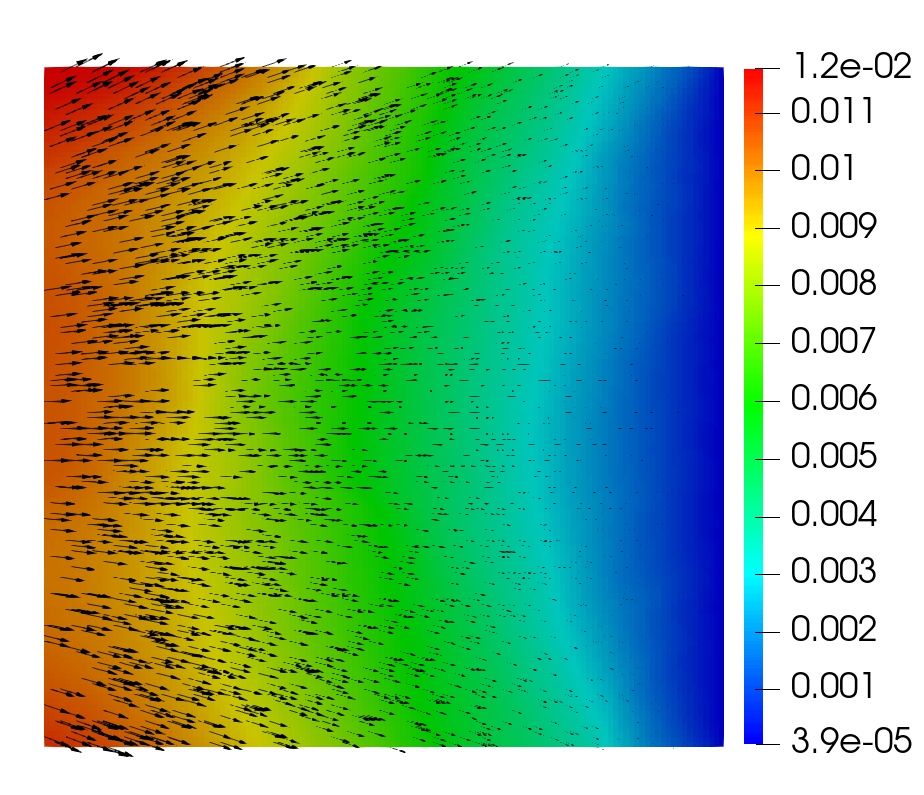}}\\
\par\end{centering}
$\quad\quad\quad\qquad\qquad$\subfloat[Stress x]{\begin{centering}
\includegraphics[width=0.33\columnwidth]{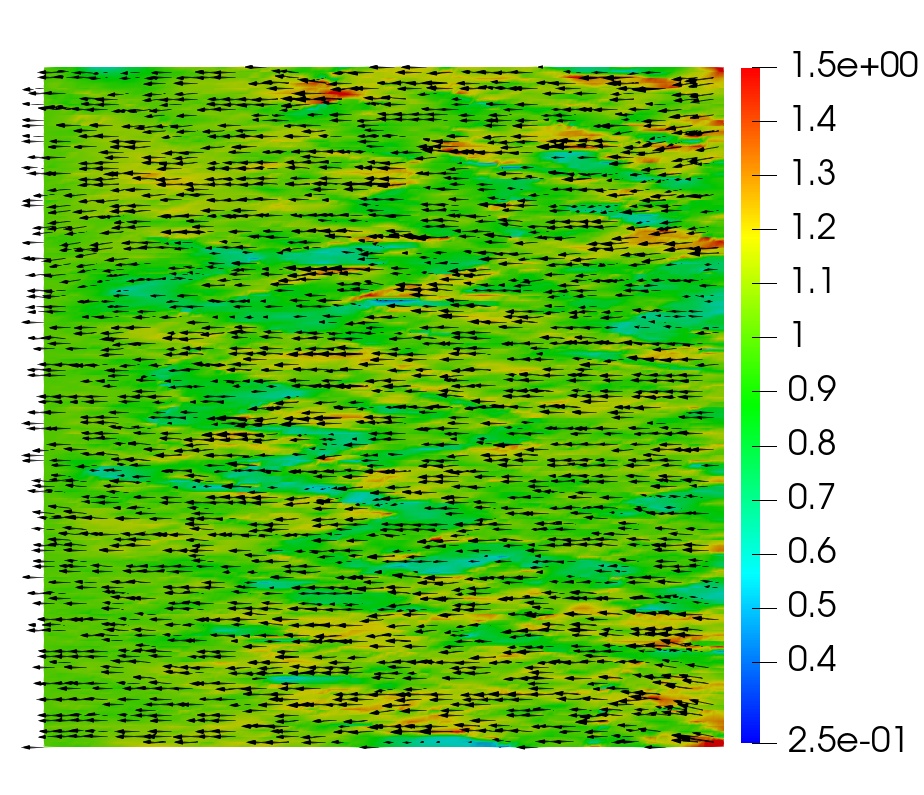}
\par\end{centering}
}\subfloat[Stress y]{
\begin{centering}
\includegraphics[width=0.33\columnwidth]{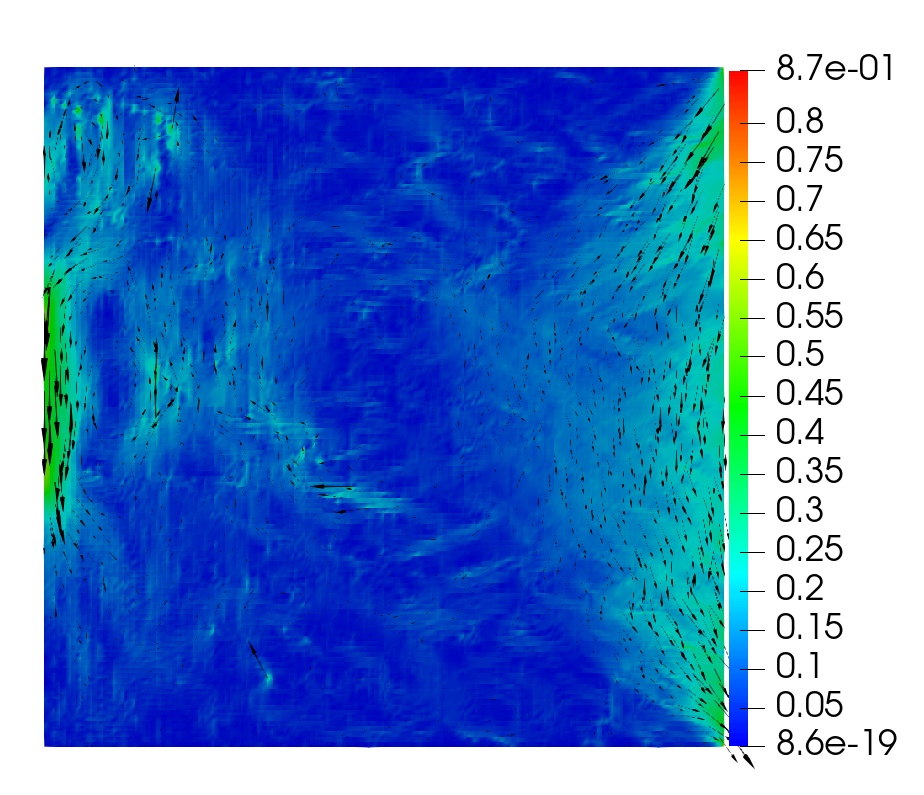}
\par\end{centering}
}

\caption{Example 3, computed solution at the final time using the monolithic domain decomposition scheme.
  \label{fig:Monolithic-Scheme:ex3}}
\end{figure}

In Table \ref{tab:Example-3,-Comparison}, we compare the average
number of interface iterations per time step in the three methods.
All three methods converge for this highly heterogeneous problem with
realistic physical parameters. While the three methods provide similar
solutions  and the number of interface iterations is comparable, the split methods are more efficient than the monolithic
method, due the less expensive CG iterations and single-physics subdomain solves. We
further note that in the split methods the Darcy interface solve is more
expensive than the elasticity one, which is likely due to the fact that the permeability
varies over seven orders of magnitude, affecting the condition number
of the interface operator.

\begin{table}[h]
  \caption{\label{tab:Example-3,-Comparison}Example 3, comparison of the number of
    interface iterations in the three methods.}

\centering{}%
\begin{tabular}{c|c|c|c|c|c}
 & \multicolumn{1}{c|}{Monolithic} & \multicolumn{2}{c|}{Drained Split} & \multicolumn{2}{c}{Fixed Stress}\tabularnewline
\hline 
$h$ & \#GMRES & \#CGElast & \#CGDarcy & \#CGElast & \#CGDarcy\tabularnewline
\hline 
$1/128$ & $565$ & $297$ & $464$ & $297$ & $464$\tabularnewline
\end{tabular}
\end{table}

\section{Conclusions}\label{sec:concl}

We presented three non-overlapping domain decomposition methods for
the Biot system of poroelasticity in a five-field fully mixed
formulation. The monolithic method involves solving an interface
problem for a composite displacement-pressure Lagrange multiplier,
which requires coupled Biot subdomain solves at each iteration. The
two split methods are based on the drained split and fixed stress
splittings. They involve two separate elasticity and Darcy interface
iterations requiring single-physics subdomain solves. We analyze the
spectrum of the monolithic interface operator and show unconditional
stability for the split methods. A series of numerical experiments
illustrate the efficiency, accuracy, and robustness of the three
methods. Our main conclusion is that while two approaches are comparable in terms of accuracy and number of interface iterations, the split methods are more computationally efficient than the monolithic method due to the less expensive CG iterations compared to GMRES and simpler single-physics subdomain solves compared to coupled Biot solves in the monolithic method.

\bibliographystyle{abbrv}
\addcontentsline{toc}{section}{\refname}
\bibliography{dd-biot}

\end{document}